\documentclass{amsart}

\usepackage{amssymb} \usepackage{amsbsy,
  amsthm, amsmath,amstext, amsopn} \usepackage[all]{xy}
\usepackage{amsfonts} \usepackage{amscd} \usepackage{stmaryrd}
 \usepackage{color}

\newtheorem{thm}{Theorem} [section]
\newtheorem{lemma}[thm]{Lemma}

\newtheorem{corollary}[thm]{Corollary}
\newtheorem{prop}[thm]{Proposition}
\newtheorem{notation}[thm]{Notation}

\theoremstyle{definition}

\newtheorem{defn}[thm]{Definition}
\newtheorem{workingdefn}[thm]{Working Definition}

\newtheorem{example}[thm]{Example}

\theoremstyle{remark}

\newtheorem{remark}[thm]{Remark}

\begin{document}

\numberwithin{equation}{section}

\newcommand{\hs}{\mbox{\hspace{.4em}}}
\newcommand{\ds}{\displaystyle}
\newcommand{\bd}{\begin{displaymath}}
\newcommand{\ed}{\end{displaymath}}
\newcommand{\bcd}{\begin{CD}}
\newcommand{\ecd}{\end{CD}}

\newcommand{\on}{\operatorname}
\newcommand{\proj}{\operatorname{Proj}}
\newcommand{\bproj}{\underline{\operatorname{Proj}}}

\newcommand{\spec}{\operatorname{Spec}}
\newcommand{\Spec}{\operatorname{Spec}}
\newcommand{\bspec}{\underline{\operatorname{Spec}}}
\newcommand{\pline}{{\mathbf P} ^1}
\newcommand{\aline}{{\mathbf A} ^1}
\newcommand{\pplane}{{\mathbf P}^2}
\newcommand{\aplane}{{\mathbf A}^2}
\newcommand{\coker}{{\operatorname{coker}}}
\newcommand{\ldb}{[[}
\newcommand{\rdb}{]]}

\newcommand{\Sym}{\operatorname{Sym}^{\bullet}}
\newcommand{\Symp}{\operatorname{Sym}}
\newcommand{\Pic}{\bf{Pic}}
\newcommand{\Aut}{\operatorname{Aut}}
\newcommand{\PAut}{\operatorname{PAut}}

\newcommand{\too}{\twoheadrightarrow}
\newcommand{\C}{{\mathbf C}}
\newcommand{\Z}{{\mathbf Z}}
\newcommand{\Q}{{\mathbf Q}}
\newcommand{\R}{{\mathbf R}}
\newcommand{\Cx}{{\mathbf C}^{\times}}
\newcommand{\Cbar}{\overline{\C}}
\newcommand{\Cxbar}{\overline{\Cx}}
\newcommand{\cA}{{\mathcal A}}
\newcommand{\cS}{{\mathcal S}}
\newcommand{\cV}{{\mathcal V}}
\newcommand{\cM}{{\mathcal M}}
\newcommand{\bA}{{\mathbf A}}
\newcommand{\bN}{{\mathbf N}}
\newcommand{\bG}{{\mathbf G}}
\newcommand{\cB}{{\mathcal B}}
\newcommand{\cC}{{\mathcal C}}
\newcommand{\cD}{{\mathcal D}}
\newcommand{\fd}{{\mathfrak d}}
\newcommand{\D}{{\mathcal D}}
\newcommand{\DD}{{\mathbb D}}
\newcommand{\cs}{{\mathbf C} ^*}
\newcommand{\boldc}{{\mathbf C}}
\newcommand{\cE}{{\mathcal E}}
\newcommand{\cF}{{\mathcal F}}
\newcommand{\bF}{{\mathbf F}}
\newcommand{\cG}{{\mathcal G}}
\newcommand{\G}{{\mathbb G}}
\newcommand{\cH}{{\mathcal H}}
\newcommand{\CI}{{\mathcal I}}
\newcommand{\cJ}{{\mathcal J}}
\newcommand{\cK}{{\mathcal K}}
\newcommand{\cL}{{\mathcal L}}
\newcommand{\baL}{{\overline{\mathcal L}}}
\newcommand{\M}{{\mathcal M}}
\newcommand{\Mf}{{\mathfrak M}}
\newcommand{\bM}{{\mathbf M}}
\newcommand{\bm}{{\mathbf m}}
\newcommand{\cN}{{\mathcal N}}
\newcommand{\theo}{\mathcal{O}}
\newcommand{\cP}{{\mathcal P}}
\newcommand{\cR}{{\mathcal R}}
\newcommand{\Pp}{{\mathbb P}}
\newcommand{\boldp}{{\mathbf P}}
\newcommand{\boldq}{{\mathbf Q}}
\newcommand{\bbL}{{\mathbf L}}
\newcommand{\cQ}{{\mathcal Q}}
\newcommand{\cO}{{\mathcal O}}
\newcommand{\cT}{{\mathcal T}}
\newcommand{\Oo}{{\mathcal O}}
\newcommand{\cY}{{\mathcal Y}}
\newcommand{\bX}{{\mathbf X}}
\newcommand{\OX}{{\Oo_X}}
\newcommand{\OY}{{\Oo_Y}}
\newcommand{\cZ}{{\mathcal Z}}
\newcommand{\otY}{{\underset{\OY}{\ot}}}
\newcommand{\otX}{{\underset{\OX}{\ot}}}
\newcommand{\cU}{{\mathcal U}}\newcommand{\cX}{{\mathcal X}}
\newcommand{\cW}{{\mathcal W}}
\newcommand{\boldz}{{\mathbf Z}}
\newcommand{\qgr}{\operatorname{q-gr}}
\newcommand{\gr}{\operatorname{gr}}
\newcommand{\rk}{\operatorname{rk}}
\newcommand{\Sh}{\operatorname{Sh}}
\newcommand{\SH}{{\underline{\operatorname{Sh}}}}
\newcommand{\End}{\operatorname{End}}
\newcommand{\uEnd}{\underline{\operatorname{End}}}
\newcommand{\Hom}{\operatorname{Hom}}
\newcommand{\bHom}{\underline{\operatorname{Hom}}}
\newcommand{\bHomY}{{\mathbf H}om_{\OY}}
\newcommand{\bHomX}{{\mathbf H}om_{\OX}}
\newcommand{\Ext}{\operatorname{Ext}}
\newcommand{\bExt}{\operatorname{\bf{Ext}}}
\newcommand{\Tor}{\operatorname{Tor}}

\newcommand{\inv}{^{-1}}
\newcommand{\airtilde}{\widetilde{\hspace{.5em}}}
\newcommand{\airhat}{\widehat{\hspace{.5em}}}
\newcommand{\nt}{^{\circ}}
\newcommand{\del}{\partial}

\newcommand{\supp}{\operatorname{supp}}
\newcommand{\GK}{\operatorname{GK-dim}}
\newcommand{\hd}{\operatorname{hd}}
\newcommand{\id}{\operatorname{id}}
\newcommand{\res}{\operatorname{res}}
\newcommand{\lrar}{\leadsto}
\newcommand{\im}{\operatorname{Im}}
\newcommand{\HH}{\operatorname{H}}
\newcommand{\TF}{\operatorname{TF}}
\newcommand{\Bun}{\operatorname{Bun}}

\newcommand{\F}{\mathcal{F}}
\newcommand{\Ff}{\mathbb{F}}
\newcommand{\nthord}{^{(n)}}
\newcommand{\Gr}{{\mathfrak{Gr}}}

\newcommand{\Fr}{\operatorname{Fr}}
\newcommand{\GL}{\operatorname{GL}}
\newcommand{\gl}{\mathfrak{gl}}
\newcommand{\SL}{\operatorname{SL}}
\newcommand{\ff}{\footnote}
\newcommand{\ot}{\otimes}
\def\Ext{\operatorname {Ext}}
\def\Hom{\operatorname {Hom}}
\def\Ind{\operatorname {Ind}}
\def\bbZ{{\mathbb Z}}

\newcommand{\nc}{\newcommand}
\nc{\ol}{\overline} \nc{\cont}{\on{cont}} \nc{\rmod}{\on{mod}}
\nc{\Mtil}{\widetilde{M}} \nc{\wb}{\overline} \nc{\wt}{\widetilde}
\nc{\wh}{\widehat} \nc{\sm}{\setminus} \nc{\mc}{\mathcal}
\nc{\mbb}{\mathbb}  \nc{\K}{{\mc K}} \nc{\Kx}{{\mc K}^{\times}}
\nc{\Ox}{{\mc O}^{\times}} \nc{\unit}{{\bf \on{unit}}}
\nc{\boxt}{\boxtimes} \nc{\xarr}{\stackrel{\rightarrow}{x}}

\nc{\Ga}{\G_a}
 \nc{\PGL}{{\on{PGL}}}
 \nc{\PU}{{\on{PU}}}

\nc{\h}{{\mathfrak h}} \nc{\kk}{{\mathfrak k}}
 \nc{\Gm}{\G_m}
\nc{\Gabar}{\wb{\G}_a} \nc{\Gmbar}{\wb{\G}_m} \nc{\Gv}{G^\vee}
\nc{\Tv}{T^\vee} \nc{\Bv}{B^\vee} \nc{\g}{{\mathfrak g}}
\nc{\gv}{{\mathfrak g}^\vee} \nc{\RGv}{\on{Rep}\Gv}
\nc{\RTv}{\on{Rep}T^\vee}
 \nc{\Flv}{{\mathcal B}^\vee}
 \nc{\TFlv}{T^*\Flv}
 \nc{\Fl}{{\mathfrak Fl}}
\nc{\RR}{{\mathcal R}} \nc{\Nv}{{\mathcal{N}}^\vee}
\nc{\St}{{\mathcal St}} \nc{\ST}{{\underline{\mathcal St}}}
\nc{\Hec}{{\bf{\mathcal H}}} \nc{\Hecblock}{{\bf{\mathcal
H_{\alpha,\beta}}}} \nc{\dualHec}{{\bf{\mathcal H^\vee}}}
\nc{\dualHecblock}{{\bf{\mathcal H^\vee_{\alpha,\beta}}}}
\newcommand{\ramBun}{{\bf{Bun}}}
\newcommand{\ramBuno}{\ramBun^{\circ}}

\nc{\Buntheta}{{\bf Bun}_{\theta}} \nc{\Bunthetao}{{\bf
Bun}_{\theta}^{\circ}} \nc{\BunGR}{{\bf Bun}_{G_\R}}
\nc{\BunGRo}{{\bf Bun}_{G_\R}^{\circ}}
\nc{\HC}{{\mathcal{HC}}}
\nc{\risom}{\stackrel{\sim}{\to}} \nc{\Hv}{{H^\vee}}
\nc{\bS}{{\mathbf S}}
\def\Rep{\operatorname {Rep}}
\def\Conn{\operatorname {Conn}}

\nc{\Vect}{{\operatorname{Vect}}}
\nc{\Hecke}{{\operatorname{Hecke}}}

\newcommand{\ZZ}{{Z_{\bullet}}}
\nc{\HZ}{{\mc H}\ZZ} \nc{\eps}{\epsilon}

\nc{\CN}{\mathcal N} \nc{\BA}{\mathbb A}
\nc{\XYX}{X\times_Y X}

\nc{\ul}{\underline}

\nc{\bn}{\mathbf n} \nc{\Sets}{{\on{Sets}}} \nc{\Top}{{\on{Top}}}

\nc{\Simp}{{\mathbf \Delta}} \nc{\Simpop}{{\mathbf\Delta^\circ}}

\nc{\Cyc}{{\mathbf \Lambda}} \nc{\Cycop}{{\mathbf\Lambda^\circ}}

\nc{\Mon}{{\mathbf \Lambda^{mon}}}
\nc{\Monop}{{(\mathbf\Lambda^{mon})\circ}}

\nc{\Aff}{{\on{Aff}}} \nc{\Sch}{{\on{Sch}}}

\nc{\bul}{\bullet}
\nc{\module}{{\operatorname{-mod}}}

\nc{\dstack}{{\mathcal D}}

\nc{\BL}{{\mathbb L}}

\nc{\BD}{{\mathbb D}}

\nc{\BR}{{\mathbb R}}

\nc{\BT}{{\mathbb T}}

\nc{\SCA}{{\mc{SCA}}}
\nc{\DGA}{{\mc DGA}}

\nc{\DSt}{{DSt}}

\nc{\lotimes}{{\otimes}^{\mathbf L}}

\nc{\bs}{\backslash}

\nc{\Lhat}{\widehat{\mc L}}

\newcommand{\Coh}{{\on{Coh}}}

\nc{\QCoh}{QC}
\nc{\QC}{QC}
\nc\Perf{\on{Perf}}
\nc{\Cat}{{\on{Cat}}}
\nc{\dgCat}{{\on{dgCat}}}
\nc{\bLa}{{\mathbf \Lambda}}

\nc{\RHom}{\mathbf{R}\hspace{-0.15em}\on{Hom}}
\nc{\REnd}{\mathbf{R}\hspace{-0.15em}\on{End}}
\nc{\colim}{\on{colim}}
\nc{\oo}{\infty}
\nc\Mod{\on{Mod}}

\nc\fh{\mathfrak h}
\nc\al{\alpha}
\nc\la{\alpha}
\nc\BGB{B\bs G/B}
\nc\QCb{QC^\flat}
\nc\qc{\cQ}

\nc{\fg}{\mathfrak g}

\nc{\fn}{\mathfrak n}
\nc{\Map}{\on{Map}} \nc{\fX}{\mathfrak X}

\nc{\ch}{\check}
\nc{\fb}{\mathfrak b} \nc{\fu}{\mathfrak u} \nc{\st}{{st}}
\nc{\fU}{\mathfrak U}
\nc{\fZ}{\mathfrak Z}
\nc{\fz}{\mathfrak z}
\nc\fk{\mathfrak k} \nc\fp{\mathfrak p}

\nc{\RP}{\mathbf{RP}} \nc{\rigid}{\text{rigid}}
\nc{\glob}{\text{glob}}

\nc{\cI}{\mathcal I}

\nc{\La}{\mathcal L}

\nc{\quot}{/\hspace{-.25em}/}

\nc\aff{\it{aff}}
\nc\BS{\mathbb S}

\nc\Loc{{\mc Loc}}
\nc\Tr{{\mc Tr}}
\nc\Ch{{\mc Ch}}

\nc\git{/\hspace{-0.2em}/}
\nc{\fc}{\mathfrak c}
\nc\BC{\mathbb C}
\nc\BZ{\mathbb Z}
\nc\BH{\mathbb H}

\nc\stab{\text{\it st}}
\nc\Stab{\text{\it St}}
\nc\dg{\text{\it dg}}
\nc\DG{\text{\it DG}}

\nc\perf{\on{-perf}}

\nc\intHom{\mathcal{H}om}

\nc\gtil{\widetilde\fg}

\nc\mon{\text{\it mon}}
\nc\bimon{\text{\it bimon}}

\newcommand{\actson}{\circlearrowright}

\title[Beilinson-Bernstein localization 
over the Harish-Chandra center]{Beilinson-Bernstein localization \\
over the Harish-Chandra center}

\author{David Ben-Zvi} \address{Department of Mathematics\\University
  of Texas\\Austin, TX 78712-0257} \email{benzvi@math.utexas.edu}
\author{David Nadler} \address{Department of Mathematics\\University
  of California\\Berkeley, CA 94720-3840}
\email{nadler@math.berkeley.edu}

\begin{abstract}
We present a simple proof of a strengthening of the derived
Beilinson-Bernstein localization theorem using the formalism of
descent in derived algebraic geometry. The arguments and results apply to arbitrary modules  
without the need to fix infinitesimal character. Roughly speaking, we
demonstrate that all $\fU\fg$-modules are the invariants, or 
equivalently coinvariants, of the action of intertwining functors (a
refined form of Weyl group symmetry) on monodromic
$\cD$-modules on the basic affine space $G/N$. This is a quantum version of
descent for the Grothendieck-Springer simultaneous resolution. 
In an appendix we present an alternative perspective, 
which identifies the descent data in both classical and quantum versions as a categorical action
of Demazure operators.

\end{abstract}

\maketitle


\tableofcontents


\section{Introduction}
The Beilinson-Bernstein localization theorem is  the
quintessential result in geometric representation theory.
To recall its statement,
fix a complex reductive group $G$, Borel subgroup $B\subset G$
with unipotent radical $N\subset B$,  and let $H = B/N$ denote the universal Cartan.
Let  $\fg, \fb, \fn$ and $\fh$ denote their  respective Lie algebras, and $W$ the Weyl group of $G$.

Fix $\lambda\in \fh^*$, and let $\cP_\lambda=\cD_{\lambda}(G/B)$
denote the dg category of $\lambda$-twisted $\D$-modules on the
flag variety.\footnote{Since we are interested in performing algebraic operations on derived categories, it is important to work in a proper homotopical setting. The term ``dg category" will stand throughout for a $\C$-linear pre-triangulated (or stable) dg category, and all operations with dg categories will be taken in the derived sense, i.e., we work in the $\oo$-category of dg categories. See Section~\ref{cat context} below for a brief summary of
our working context.}  By $\lambda$-twisted  $\D$-modules, one can take $\D$-modules on the basic affine space $G/N$ which
are weakly $H$-equivariant and have   monodromy $\lambda$
along the $H$-orbits. (The notation $\cP_\lambda$ reflects 
the analogy with principal series representations of real or $p$-adic
groups.) 

Let $\fU\fg$ denote the universal enveloping algebra of $\fg$, and $\fZ\fg  \simeq \cO(\fh^*\git W) \subset \fU\fg$ the Harish-Chandra center.
Let $\fU\fg\module_{[\lambda]}$ denote the
dg category of $\fU\fg$-modules with infinitesimal character
the $W$-orbit $[\lambda]\in \fh^*\git W\simeq\Spec \fZ\fg$.

\begin{thm}[\cite{BB}]\label{bb thm} For $\lambda\in \fh^*$  regular,
localization and global sections  provide inverse equivalences of
derived categories
$$\xymatrix{
\Delta:\fU\fg\module_{[\lambda]} \ar@{<->}[r]^-\sim & \cP_\lambda:\Gamma
}
$$

\end{thm}

\begin{remark}[Refinements] 
There are several standard refinements of the above statement.

1) One can bound the cohomological amplitude of $\Gamma$ by the length
of the Weyl group element $w\in W$ for which $w\cdot\lambda$ is dominant.
In particular, when $\lambda $ is dominant, the equivalence preserves
the corresponding abelian categories. 


2) For $\lambda\in \fh^*$  singular, one can exhibit $\fU\fg\module_{[\lambda]}$ as the quotient of
$\cP_\lambda$ by the kernel of $\Gamma$. Alternatively, one can realize $\fU\fg\module_{[\lambda]}$ as the subcategory
of $\cP_\lambda$ left-orthogonal to the kernel of $\Gamma$. 
These discrepancies can be bridged
via localization on {partial} flag
varieties ~\cite{Kashiwara,BMR1,BK}.

3) The theorem respects the natural $G$-actions and so identifies equivariant objects on each side. Thus for a subgroup
$K\subset G$, it provides an equivalence between Harish-Chandra $(\fU\fg\module_{[\lambda]}, K)$-modules
and $K$-equivariant  $\lambda$-twisted $\D$-modules.

4) One can naturally extend the theorem to generalized monodromy and infinitesimal character and thus obtain
an equivalence not for fixed  regular $\lambda \in \fh^*$ but over its formal neighborhood $\lambda^\wedge$ within $\fh^*$. Namely,  let $\cP_{\lambda^\wedge}=\cD_{\lambda^\wedge}(G/B)$ denote the dg category of weakly $H$-equivariant $\D$-modules on $G/N$ with monodromy in $\lambda^\wedge$. Similarly, let
  $\fU\fg\module_{[\lambda^\wedge]}$  denote the
dg category of $\fU\fg$-modules with infinitesimal character in the $W$-orbit
 $[\lambda^\wedge] $ or equivalently formal neighborhood $[\lambda]^\wedge$. Then  when $\lambda$ is regular,
localization and global sections  provide inverse equivalences
$$\xymatrix{
\Delta:\fU\fg\module_{[\lambda^\wedge]} \ar@{<->}[r]^-\sim & \cP_{\lambda^\wedge}:\Gamma
}
$$
 Alternatively,
one can equip the original statement of  Theorem~\ref{bb thm} with the symmetries that control the extension to the formal neighborhood. 
Namely, one can pass to (co)modules for the (co)monad of the adjunction given by restriction and extension of modules along the inclusion of $\lambda$
into  $\lambda^\wedge$.
\end{remark}

Despite its dramatic importance, which is hard to overestimate, there nevertheless remain several deficiencies of Theorem \ref{bb thm},  even after taking into account the above refinements. 

First, and foremost of our motivations, the theorem
applies to modules with {fixed}, or at most generalized,
infinitesimal character. This makes it difficult to apply to or even compare with questions
of harmonic analysis, which typically involve the geometry or topology
of families of representations (for example, the Plancherel formula and
Baum-Connes conjecture).  

Second, it is unnatural and thus often confusing that the parameters on the two sides of the theorem do not match: to
localize modules  with infinitesimal character $[\lambda]\in
\fh^*\git W$, we must choose a representative lift $\lambda\in \fh^*$. It is worth noting the impact of this choice:
the geometry of the corresponding $\D$-modules -- most basically, the
dimension of their support -- depends on this choice.

Third, the
theorem does not apply as stated to singular infinitesimal character
where the category of $\D$-modules is larger than the corresponding
category of modules. As mentioned above, one can bridge this discrepancy
via localization on {partial} flag
varieties.  But this seems to only increase the distance from a uniform statement over all infinitesimal characters.

In this paper, we present a natural refinement of the derived
Beilinson-Bernstein localization theorem which simultaneously corrects
these three drawbacks by invoking descent. A naive paraphrase of the main
result asserts: ``the category of $\fU\fg$-modules is  equivalent to the  Weyl group
invariants in the category of all monodromic $\D$-modules on $G/N$''.

In what immediately follows, we first
state a precise version of this result, then explain the paraphrase and a ``classical
limit'' for quasicoherent sheaves on the Grothendieck-Springer
resolution.


\subsection{Beilinson-Bernstein via Barr-Beck} 
A natural setting  for simultaneously localizing over
all infinitesimal characters is 
the dg category $\cP=\D_{\mon}(G/N)$ of all weakly
$H$-equivariant $\D$-modules on the basic affine space $G/N$
(see
for example \cite{Beilinson ICM}). The subscript ``$\mon$" stands for the term monodromic, which we will understand 
to be a synonym for weakly $H$-equivariant, and in particular not imply specified monodromy along the $H$-orbits. 
(The notation $\cP$ reflects 
the analogy with the universal principal series representation of a real group or
 the universal unramified
principal series representation of a $p$-adic group.)

The natural right $H$-action on $G/N$ induces an identification of $\fU\fh\simeq \cO(\fh^*)$ with central $H$-invariant differential operators on $G/N$. 
This equips $\cP$ with a linear structure over the base $\fh^*$ with fiber  at $\lambda\in \fh^*$ the previously encountered dg category $\cP_\lambda = \D_\lambda(G/B)$ of  $\lambda$-twisted $\D$-modules.
Similarly, restricting to the formal neighborhood $\lambda^\wedge$, we  recover 
the dg category
$\cP_{\lambda^\wedge}=\cD_{\lambda^\wedge}(G/B)$   of weakly $H$-equivariant $\D$-modules with monodromy in $\lambda^\wedge$.

The natural left $G$-action on $G/N$ induces an embedding of $\fU\fg$ in $H$-invariant differential operators on $G/N$. 
This gives rise to an adjunction
$$\xymatrix{
\Delta:\fU\fg\module\ar@{<->}[r]  & \cP: \Gamma}
$$ 
between the dg category of all $\fU\fg$-modules and $\cP$. Note that here $\Gamma$ denotes the $H$-invariants in the usual global sections functor. The functors intertwine the linear structure over $\fh^*\git W$ on the left and that over $\fh^*$ on the right.

An important observation is that the adjunction is
naturally {\em ambidextrous}: $\Delta$ is canonically
 the {right} adjoint of $\Gamma$ as well. This
is a reflection of the Calabi-Yau structure of  $\cP$,
 the key ingredient in the approach of \cite{BMR1,BMR2} to
establishing Beilinson-Bernstein equivalences (following an argument
of \cite{BKR} in the setting of the McKay correspondence, and extended
to localization for quantum symplectic resolutions in \cite{McN}).

We will refer to the composition 
$$\xymatrix{
\cW=\Delta\circ\Gamma} \in \End(\cP)
$$  as the {\em Weyl functor}. 
It naturally comes equipped with the structure
of a monad and a comonad, or in other words, an algebra
and coalgebra object in endofunctors
of $\cP$. In fact, one could organize all of the structure in what could rightfully be called a {\em Frobenius monad}.
We will use this term as evocative shorthand for the monad of an ambidextrous adjunction (and will not attempt 
to independently formalize it in the $\oo$-categorical setting; see Section~\ref{cat context} below for further discussion
and references to the discrete setting).

Another crucial feature of $\Delta$ is its conservativity: no
non-zero objects localize to zero (an easy consequence of the general localization
formalism). We find ourselves in the setting of the Barr-Beck theorem, in its $\oo$-categorical form due to Lurie.
Recall it states that given an adjunction $F:\cC\longleftrightarrow \cD:G$ between $\oo$-categories, if 
the right adjoint $G$ is conservative, and any $G$-split simplicial object of $\cD$ admits a colimit preserved by $G$, then
$G$ exhibits $\cD$ as monadic over $\cC$.
This allows us to describe $\fU\fg\module$ in terms of $\cP$ and the Weyl functor
$\cW$ with its natural structures.  
 Let us write $\cP_\cW$ and $\cP^\cW$ for the respective dg categories of $\cW$-modules and $\cW$-comodules
in $\cP$.

\begin{thm}[Barr-Beck version of localization] \label{monadic BB}
%
 There are canonical equivalences
$$\cP_\cW\simeq \fU\fg\module \simeq \cP^\cW$$ between the dg categories of
  $\cW$-modules and $\cW$-comodules in $\cP$ and the category of
  $\fU\fg$-modules.
\end{thm}

\begin{remark}For fixed regular
$[\lambda]\in\fh^*\git W$, the arguments of \cite{BMR1} utilize the ambidextrous property and the indecomposability of $\D_\lambda(G/B)$ to
deduce the Beilinson-Bernstein equivalence.  For singular parameters or
 families of infinitesimal characters, the indecomposability fails, but we recover $\fU\fg\module$
 as a summand of $\cP$ in the form of an isotypic component $\fU\fg\module=\cP^\cW$.
\end{remark}

\begin{remark}[``Reverse" Beilinson-Bernstein and the nil-Hecke algebra]
As we explain in Section~\ref{nil section}, we can also run Barr-Beck in the opposite direction, accessing $\cP$ through the global sections functor. This is no longer conservative, so the result is a description of {\em globally generated} monodromic $\cD$-modules on $G/N$: 
we have equivalences
$$\cP^{glob}\simeq \wt{\fU\fg}\module$$
where $\wt{\fU\fg}=\fU\fg \otimes_{\fZ\fg} \fU\fh$ is the extended enveloping algebra, and
$$(\cP^{glob})^\BH\simeq \fU\fg\module$$
where $\BH$ is the nil-Hecke algebra, the $\fU\fh$-algebra generated by the Demazure divided-difference operators, which controls descent along the map $$\fz:\fh^\ast\to \fh^\ast//W$$ from the dual Cartan to its coarse quotient by the Weyl group.
Thus we see that the Weyl monad $\cW$ in Theorem~\ref{monadic BB} combines the action of the nil-Hecke algebra with projection along the kernel of the global sections functor. See also~\ref{Demazure intro} for a closely related appearance of the Demazure operators.
\end{remark}


\subsection{Beilinson-Bernstein via Hecke symmetry}
In order to exploit the formal assertion of Theorem~\ref{monadic BB}, we need to
identify the Weyl functor $\cW$ concretely in terms of the symmetries of $\cP$. 
By construction, the category $\cP$ carries two fundamental commuting actions.  

On the one hand,
the left $G$-action on $G/N$ naturally equips $\cP$ with the structure of 
 {\em de Rham
  $G$-category}.
 By  this, we mean an algebraic $G$-action that is   
 infinitesimally trivialized, or in other words, the induced action of the formal group of $G$ is trivialized.
This can be formalized by saying that $\cP$ is a 
module for the monoidal
dg category $\D(G)$ of $\D$-modules on $G$ under  convolution.

On the other hand, the Hecke category $\cH=
\D_{\bimon}(N\bs G/N)$ of bimonodromic $\D$-modules on $N\bs G/N$ naturally acts on $\cP$ on the right by intertwining functors. 
In the same way that $\cP$ is linear over the base $\fh^*$, the Hecke category $\cH$ is linear over the base $\fh^*\times\fh^*$,
and its monoidal structure is compatible with convolution of bimodules over  $\fh^*\times\fh^*$. 

One can view the Hecke category $\cH$ as the monodromic generalization 
 of the familiar finite Hecke category $\D(B\bs G/B)$, which in turn is a categorical
form of the finite Hecke algebra or Artin braid group.
More generally, one finds the Hecke category $\cH_{\lambda, \mu} = \D_{\lambda, \mu}(B\bs G/B)$ of $\lambda, \mu$-twisted $\D$-modules
as the fiber of $\cH$ at a point $(\lambda, \mu)\in \fh^*\times\fh^*$.
Similarly, restricting to the 
the formal neighborhood 
$\lambda^\wedge\times \mu^\wedge$,
one finds
the Hecke category $\cH_{\lambda^\wedge, \mu^\wedge} = \D_{\lambda^\wedge, \mu^\wedge}(B\bs G/B)$ 
of weakly $H\times H$-equivariant $\D$-modules with monodromy in $\lambda^\wedge\times \mu^\wedge$.

\begin{remark}
At first glance, it might look easier to work with strictly  $\lambda$-twisted $\D$-modules rather than generalized twisted $\D$-modules with monodromy in $\lambda^\wedge$. But since $\lambda^\wedge$ is flat over the base $\fh^*$ unlike the point $\lambda$ itself, tensor products and ultimately convolution patterns restrict to $\lambda^\wedge$ in a less intricate way. Alternatively, one can equip strictly $\lambda$-twisted $\D$-modules with the symmetries that control their extension to $\lambda^\wedge$.
Namely, one can pass to (co)modules for the (co)monad of the adjunction given by restriction and extension of modules along the inclusion of $\lambda$
into  $\lambda^\wedge$.
But keeping track of this extra Koszul dual structure can be less intuitive than simply working over $\lambda^\wedge$.
\end{remark}

\begin{remark} Here are some simple observations to help orient the reader.  

The restriction $\cH_{\lambda^\wedge, \mu^\wedge}$ vanishes unless $\mu = w\cdot (\lambda+ \lambda')$, for some Weyl group element $w\in W$ and integral weight $\lambda'\in \Lambda^*\subset \fh^*$.
The restrictions $\cH_{\lambda^\wedge, \mu^\wedge}$ and $\cH_{\lambda'^\wedge, \mu'^\wedge}$ are non-canonically equivalent if there are  Weyl group
elements $w, w' \in W$ such that $\lambda' = w\cdot \lambda, \mu' = w'\cdot\mu$. 
The restrictions $\cH_{\lambda^\wedge, \mu^\wedge}$ and $\cH_{\lambda'^\wedge, \mu'^\wedge}$ are canonically equivalent if the difference of parameters
is integral $\lambda-\lambda',\mu-\mu' \in\Lambda^*\subset \fh^*$. 

The monoidal structure of $\cH$ descends to compatible compositions 
$$
\xymatrix{
\cH_{\lambda^\wedge, \mu^\wedge} \otimes \cH_{\mu^\wedge, \nu^\wedge} \ar[r] & \cH_{\lambda^\wedge, \nu^\wedge}
}$$
Moreover, the restriction maps $\cH \to \cH_{\lambda^\wedge, \mu^\wedge}$ are compatible with the above restricted
compositions. In particular, the diagonal restriction map $\cH \to \cH_{\lambda^\wedge, \lambda^\wedge}$ is a monoidal map
of monoidal categories.
\end{remark}

\begin{remark}
A result of \cite{character}
asserts that for any fixed $\lambda \in\fh^*$, the diagonally $\lambda$-twisted Hecke categories $\cH_{\lambda, \lambda} = \D_{\lambda, \lambda}(B\bs G/B)$, $\cH_{\lambda^\wedge, \lambda^\wedge} = \D_{\lambda^\wedge, \lambda^\wedge}(B\bs G/B)$   are
both  categorified analogues of finite
dimensional semisimple Frobenius algebras: they are the values on a
point of extended oriented two-dimensional topological field
theories. More precisely, they are  two-dualizable Calabi-Yau algebra
objects in dg categories.
\end{remark}

Along the way, as a simple application of results of~\cite{character}, we will establish the following basic relationship between the above commuting symmetries.

\begin{thm}\label{Hecke Morita}
There is a monoidal equivalence $$\xymatrix{\cH\simeq \End_{\D(G)}(\cP)}$$
between the Hecke category $\cH$ and  $\D(G)$-linear endofunctors of $\cP$.
\end{thm}

\begin{remark}
Specializing over subsets of the base $\fh^*$ immediately provides analogous assertions for monodromic $\D$-modules with prescribed monodromies.
\end{remark}

Observe that the adjoint $G$-action on $\fg$ naturally equips $\fU\fg\module$ with the structure of de Rham $G$-category.
As was emphasized in \cite{BD} (see also
\cite{FG}), the localization and global sections functors commute with
the corresponding de Rham $G$-actions. 
Therefore, or as can be verified independently, the Weyl functor $\cW$ must be represented by an object of the Hecke category $\cH$ which we will also denote by $\cW$. It turns out to have a simple geometric description.

\begin{lemma}\label{universal Weyl sheaf} The {\em universal Weyl sheaf}  $\cW\in \cH$, the Hecke kernel for the Weyl functor, is the 
sheaf of differential operators on $N\bs G/N$ with its canonical weakly
$H$-biequivariant structure.
\end{lemma}

\begin{remark}
Differential operators on the stack $N\bs G/N$ can be calculated by starting with differential operators on $G/N$ and then performing quantum Hamiltonian reduction for the left $N$-action.  Thus one  quotients differential operators on $G/N$ by the vector fields generating the left  $N$-action and then takes $N$-invariants. 
\end{remark}

\begin{thm}[Hecke version of localization] \label{Hecke BB}

The universal Weyl sheaf $\cW$ is naturally an algebra and coalgebra in $\cH$.
There are canonical equivalences
$$\cP_\cW\simeq \fU\fg\module \simeq \cP^\cW$$ between the categories of
  $\cW$-modules and $\cW$-comodules in $\cP$ and the category of
  $\fU\fg$-modules.
\end{thm}

 Altogether, one can view the universal Weyl sheaf $\cW$ as  a categorical idempotent, providing a lift of the idempotent in  the groupoid algebra governing descent along the 
quotient map $\fh^* \to \fh^*\git W$.
It is clarifying to reverse our momentum and regard  it as a family of idempotents
 as we vary the infinitesimal character.
We include here a brief  discussion of the two most extreme cases.
 
 \begin{remark}
 It is  worth pointing out first of all that while $\cW \in\cH$ is the sheaf of differential operators on $N\bs G/N$ so ``large",
 its fibers $\cW_{\lambda, \mu} \in \cH_{\lambda, \mu}$ are  regular holonomic twisted $\D$-modules on $B\bs G/B$ so ``small". This is a twisted instance of the
 general fact
 that on a quotient stack $H\bs X$ such that $X$ is smooth and $H$ is affine algebraic acting with finitely many  orbits in $X$,
 the sheaf of differential 
 operators on $H\bs X$ is regular holonomic. In fact, any coherent $\D$-module on $H\bs X$  is regular holonomic.
\end{remark}

At one extreme, when $\lambda \in \fh^*$ is regular,
 consider the direct sum of the potential targets for  localization 
$$
\xymatrix{
\cP_{[\lambda^\wedge]}=\bigoplus_{w\in W} \cP_{w\cdot\lambda^\wedge}
}$$
and the corresponding Hecke category
$$
\xymatrix{
\cH_{[\lambda^\wedge]} =\bigoplus_{w, w'\in W} \cH_{w\cdot\lambda^\wedge, w'\cdot\lambda^\wedge}
}
$$
The classical Beilinson-Bernstein theorem asserts that $\Gamma$ is already a derived equivalence (and in fact t-exact up to a shift) when restricted to any summand. Moreover the classical action of principal series intertwining functors intertwine the different localizations~\cite{Beilinson ICM}. Thus we can understand the statement of Theorem~\ref{Hecke BB} as saying we perform all $|W|$ localizations and keep track of the Weyl group action. Let us spell this out.

First, for $\lambda \in \fh^*$ generic (with respect to the notion of integrality coming from coroot functionals), 
we have that
$$
\xymatrix{
\cH_{w\cdot \lambda^\wedge, w'\cdot\lambda^\wedge} \simeq \cQ(\lambda^\wedge),
&
\mbox{for all $w, w'\in W$,}
}
$$
where $\cQ(\lambda^\wedge)$ denotes the dg category of quasicoherent sheaves on the formal neighborhood of $\lambda \in\fh^*$ or equivalently that of $[\lambda]\in \fh^*\git W$. (Since $\lambda$ is generic, only one Bruhat double-coset in $G$ supports any  bimonodromic modules with parameters in $w\cdot \lambda^\wedge, w'\cdot\lambda^\wedge$,
and all such modules are only constrained by the fact that the parameters must be related by $w'w^{-1} $.)
The monoidal structure of  $\cH_{[\lambda^\wedge]}$  is simply that of the groupoid algebra of the regular Weyl groupoid 
$$\xymatrix{
W \times W \ar@<0.7ex>[r] \ar@<-0.7ex>[r] & W
}
$$ 
with scalars in the constant tensor category
$\cQ(\lambda^\wedge)$. 
Furthermore, the restriction $\cW_{[\lambda^\wedge]}\in \cH_{[\lambda^\wedge]}$ of the universal Weyl sheaf corresponds to the direct sum of the structure sheaf $\cO_{\lambda^\wedge}\in \cQ(\lambda^\wedge)$  in each factor. 
Thus it provides the categorical analogue of the constant
idempotent in the groupoid algebra of $W$. As expected, its modules and comodules in $\cP_{[\lambda^\wedge]}$ are equivalent to a single copy of $\cP_{\lambda^\wedge}$.

Leaving behind the generic case,
for $\lambda \in \fh^*$ regular integral, the 
 Hecke category
$
\cH_{[\lambda^\wedge]} 
$
now contains all of 
the combinatorics of Kazhdan-Lusztig theory. A common approach to capture this structure is to focus on the exceptional collections of standard or costandard objects associated to Schubert cells. Their convolutions provide dual realizations of the  groupoid algebra of the natural Artin braid group action on the Weyl group
$$\xymatrix{
B_W \times W \ar@<0.7ex>[r] \ar@<-0.7ex>[r] & W
}
$$ 
with scalars in the constant tensor category
$\cQ(\lambda^\wedge)$. 
In fact, one can embed the 
regular Weyl groupoid 
inside the above braid groupoid by choosing  appropriate lifts of minimal standard or costandard objects depending on
whether multiplication 
increases or decreases length. 
With this observation in hand, the restriction $\cW_{[\lambda^\wedge]}\in \cH_{[\lambda^\wedge]}$  corresponds
to a direct sum of specified standard and costandard objects giving the constant idempotent of $W$.
 Thus as expected, its modules and comodules in $\cP_{[\lambda^\wedge]}$ are equivalent to $\fU\fg\module_{[\lambda^\wedge]} $ in the form of a single copy of $\cP_{\lambda^\wedge}$.

\begin{remark}[Borel-Weil-Bott]
The above picture for  $\lambda \in \fh^*$ regular integral gives a nice framework for understanding the Borel-Weil-Bott theorem. Let us focus on the objects of $\cP_{[\lambda^\wedge]}$ given by the shifted line bundles $\cO(w\cdot \lambda)[\ell(w)] \to G/B$, where $\ell(w)$ denotes the length of $w\in W$. Their global sections are  the irreducible $G$-representation $V_{[\lambda]}$ placed in degree zero thanks to their original shifts by length. The summands of the idempotent
$\cW_{[\lambda^\wedge]}\in \cH_{[\lambda^\wedge]}$ intertwine the shifted line bundles, increasing or decreasing degree depending on what happens to length.

Equivalently, we can apply the $G$-equivariant description of Theorem~\ref{Hecke BB} to $G$-integrable representations $(\fU\fg\module)^G\simeq Rep(G)$, obtaining an equivalence with $\cW$-comodules in $(\cP)^G\simeq \cD(pt/N)^H\simeq Rep(H)$.
The comonad $\cW$ is thereby identified with the Borel-Weil-Bott comonad on $Rep(H)$, given by the standard parabolic induction/restriction adjunction.
\end{remark}

Finally, as $\lambda \in \fh^*$ specializes to become singular, the geometry becomes more interesting,
 ultimately reflecting the nontrivial scheme structure of the quotient map $\fh^* \to \fh^*\git W$
 along its ramification locus .
In the most singular case $\lambda=0 \in \fh^*$, 
recall that the target of localization $\cP_{[0^\wedge]} \simeq \cP_{0^\wedge}$ is dramatically different from
the singular category $\fU\fg\module_{[0^\wedge]}$.
The Hecke category $\cH_{[0^\wedge]} = \cH_{0^\wedge, 0^\wedge}$ controlling this difference can be viewed in dual standard and costandard ways as a categorical analogue of the Artin braid group algebra
 (rather than groupoid algebra as  above).
 
 To concretely discuss the restriction $\cW_{[0^\wedge]} \in \cH_{[0^\wedge]}$, let us focus on its convolution against the object
 $ \cO(-\rho)\in \cP_{[0^\wedge]}$ given by the canonical bundle of $G/B$. One can calculate that 
 $$
 \xymatrix{
 \cT = \cW_{[0^\wedge]} * \cO(-\rho)
 }
 $$
 is the tilting sheaf given by the projective cover of the skyscraper at the closed Schubert cell. It is well known
 that  $\cT$  governs the singular category (as in the work of Soergel~\cite{garben}); for example, the kernel of the global
 sections functor to $\fU\fg\module_{[0^\wedge]}$ is its right-orthogonal.


\subsection{Comparison with $K$-theory}
We would like to interpret Theorem \ref{Hecke BB} as a refinement of results in $K$-theory which are
 closer in form to the naive paraphrase:
``the category of $\fU\fg$-modules is  equivalent to the  Weyl group
invariants in the category of all monodromic $\D$-modules on $G/N$''.

First, let us recall $K$-theory versions of the Weyl character formula and Borel-Weil-Bott
theorem from~\cite{BH} (where the context is equivariant $KK$-theory 
and the results are closely related
to the Baum-Connes conjecture for Lie groups).  To proceed in this setting, let $G_c\subset G$ be a maximal compact subgroup,
and $T_c\subset G_c$ a maximal torus, so that we have $G/B \simeq G_c/T_c$.
 Natural morphisms relate the
representation ring of $G_c$ and the equivariant
$K$-theory of the flag manifold 
$$ \xymatrix{
\Delta: K_{G_c}^*(pt) \ar@{<->}[r] & K_{G_c}^*(G_c/T_c) : \Gamma}
$$
where $\Gamma$ is the equivariant index (Borel-Weil-Bott construction)
and $\Delta$ is given by pullback followed by multiplication by the
virtual bundle $\Omega^\cdot$. Note that for a representation $V\in
K_{G_c}^*(pt)$, the virtual bundle $\Delta(V)$ can be identified with
the complex that fiberwise computes $\fn$-homology (where $\fn$ is the unipotent radical
of the  stabilizer of a point in $G/B$). In other
words,  the virtual bundle $\Delta(V)$ is the $K$-theory image of the Beilinson-Bernstein
localization of $V$.

The Weyl group $W$ acts (non-holomorphically)  from the right on
$G_c/T_c$, and the main theorem of ~\cite{BH} calculates (in the context of equivariant $KK$-theory) that $\Gamma\circ \Delta=|W|\on{Id}$,
and when restricted to the $W$-invariants,
$\Delta\circ \Gamma = |W| \on{Id}.$ Identifying the
virtual character of $\Omega^\cdot$ with the Weyl denominator, one
recovers the Weyl character formula. In fact,  the main theorem of ~\cite{BH} calculates that 
 $\Delta\circ\Gamma$ is given by the standard idempotent in
the group algebra of $W$ (the projector to the trivial
representation), realized as a sum of standard intertwining operators.
Note that the group algebra $\BC W$ is a Frobenius algebra, and so $W$-invariants and coinvariants are both identified with the summand given by the
image of the standard idempotent.

In our present categorical setting, the action of the Weyl group  on
 equivariant $K$-theory  is replaced by the action of the Hecke category
$\cH=\D_{\bimon}(N\bs G/N)$ on the category $\cP = \D_{\mon}(G/N)$. For fixed integral $\lambda\in \fh^*$, 
the corresponding Hecke category $\cH_{\lambda, \lambda}=\D_{\lambda, \lambda}(\BGB)$ of
Kazhdan-Lusztig theory has Grothendieck group the group algebra $\BC W$. But in fact  standard bases
 of $\cH_{\lambda, \lambda}$ provide
actions on categories not of the Weyl group $W$ but of the corresponding Artin braid group. 
The Frobenius monad $\cW$ is the categorified analogue of the standard
idempotent, with $\cW$-modules playing the role of $W$-invariants and
$\cW$-comodules that of $W$-coinvariants.


\subsection{Classical limit}
The preceding description of Beilinson-Bernstein localization has a natural classical
analogue for quasicoherent sheaves on the Springer resolution. Its
interpretation as a form of proper descent becomes evident in this setting.
We will proceed in the context of derived algebraic geometry, and in particular that of perfect stacks as introduced in~\cite{BFN}. 
In particular, let $\pi:X\to Y$ denote any morphism of perfect stacks, and $\XYX$ the corresponding derived fiber product. 

First, consider the symmetric monoidal dg category $\cQ(Y)$
of quasicoherent sheaves equipped with tensor product, and its natural module dg category $\cQ(X)$ under the pullback $\pi^*$. The dg category
$\cQ(\XYX)$  is monoidal with respect to convolution,  acts on $\cQ(X)$ by endofuntors
with a natural $\cQ(Y)$-linear structure,
thus leading to a monoidal equivalence 
$$\xymatrix{
\cQ(\XYX)\ar[r]^-\sim &
\on{End}_{\cQ(Y)}(\cQ(X)).
}$$ 

The adjunction $(\pi^*,\pi_*)$ on quasicoherent sheaves associated to $\pi:X\to Y$ defines a
comonad $T^\vee=\pi^*\pi_*$ acting on $\cQ(X)$, and it is easy to see that
$T^\vee$ is represented by the coalgebra object 
$$\cA=\cO_{\XYX}\in
\cQ(\XYX).$$ 
Observe that $\cA$ is simply the groupoid coalgebra (functions on the
groupoid with convolution coproduct) for the descent groupoid $\XYX$
acting on $X$.
When
$\pi$ is faithfully flat, then descent holds by \cite[7]{dag11},  providing an equivalence $$\cQ(Y)\simeq \cQ(X)^{\cA}.$$
Such flatness will not hold in our setting, but we will know that $\pi^*$ is conservative and cocontinuous and thus
descent holds by the Barr-Beck-Lurie theorem.

Next, let us assume that $\pi$ is proper and surjective on
field points.
Consider the  symmetric monoidal dg category $\cQ^!(Y)$
of ind-coherent sheaves equipped with $!$-tensor product, and its natural module dg category $\cQ^!(X)$ under the pullback $\pi^!$. 
The dg category
$\cQ^!(\XYX)$  is monoidal with respect to convolution,  acts on $\cQ^!(X)$ by endofuntors
with a natural $\cQ^!(Y)$-linear structure,
thus leading to a monoidal functor which is typically not an equivalence 
$$\xymatrix{
\cQ^!(\XYX)\ar[r] &
\on{End}_{\cQ^!(Y)}(\cQ^!(X)).
}
$$

The adjunction $(\pi_*,\pi^!)$ on ind-coherent sheaves associated to $\pi:X\to Y$ defines a
monad $T=\pi^!\pi_*$ acting on $\cQ^!(X)$, and
it is easy to see that  $T$ is represented by an algebra object
$$
\xymatrix{\cA^!=\omega_{\XYX/X}\in \cQ^!(\XYX)}.
$$ 
Observe that $\cA^!$ is the groupoid algebra (relative volume forms on the
groupoid with convolution product) for the descent groupoid $\XYX$
acting on $X$.
The proper descent theorem of \cite[Proposition A.2.8]{preygel} and \cite[7.2.2]{indcoh} provides an equivalence
$$\cQ^!(Y)\simeq \cQ^!(X)_{\cA^!}.$$

Finally, let us now assume that $X$ and $Y$ are both smooth, so that we have
canonical identifications $\cQ(X)\simeq \cQ^!(X)$ and $\cQ(Y)\simeq
\cQ^!(Y)$. Moreover, let us assume that $\pi$ is proper, surjective on field points, and {\em crepant}, or in other words, Calabi-Yau
of dimension zero in that we are given a trivialization of its
relative dualizing sheaf and thus an identification $\pi^*\simeq
\pi^!$. In particular, this implies that $\pi^* \simeq \pi^!$ are simultaneously continuous and cocontinuous, as well as conservative. Then $T \simeq T^\vee$ is a Frobenius monad, the
groupoid algebra $\omega_{\XYX/X} \simeq \cO_{\XYX}$ is a Frobenius algebra object, and there are
 equivalences
$$\cQ(X)^T\simeq \cQ(Y)\simeq \cQ(X)_T.$$

The prime example for these restrictive hypotheses is the
Grothendieck-Springer simultaneous resolution $\pi:\tilde\fg\to
\fg$ with descent groupoid the Grothendieck-Steinberg variety. The resulting
descent picture is precisely the classical limit of
Beilinson-Bernstein localization. We develop the details of this  in
Section \ref{classical}, in particular its various
specializations over different regions in the
adjoint quotient.

\subsection{Demazure descent}\label{Demazure intro}
In the Appendix, Section~\ref{Demazure section}, we present a different picture of our results as a categorical form of the action of Demazure operators, inspired by~\cite{AK1,AK2}. Namely we show how the language of ind-coherent sheaves on inf-schemes~\cite{GR} allows a simple uniform description of the monads controlling both the classical Grothendieck-Springer descent and the quantum Beilinson-Bernstein localization. They both describe arise as images of the Demazure monad $\fd$, the algebra in the Demazure Hecke category $\DD=(\cQ^!(\BGB),\ast)$ whose action on a $\DD$-module projects onto $\DD$-invariants. As a result we can interpret both the classical and quantum equivalences as a categorical form of taking invariants for the Demazure divided-difference operators.

\subsection{Categorical context} \label{cat context}

We will work throughout in the language of derived algebraic geometry
following \cite{topos,HA}; we refer the reader to \cite{BFN,DG}
for some gentle discussion of this context and its basic tools.
We will work throughout over the complex numbers $\BC$.

The words ``category'' and ``dg category" will stand for either a
$\BC$-linear pre-triangulated dg category or a $\BC$-linear stable
$\oo$-category, and we refer to \cite{DG} for a general homotopical treatment of dg categories and~\cite{cohn} for an explicit comparison of the homotopy theories of dg categories and stable $\infty$-categories.
Such categories fit into two related contexts: 1) $\DG_\BC$ the
symmetric monoidal $\oo$-category of stable
presentable $\BC$-linear dg categories with morphisms colimit preserving 
functors, and 2) $\dg_\BC$  the
symmetric monoidal $\oo$-category of small stable idempotent-complete $\BC$-linear
dg-categories with morphisms exact functors. We say that a functor
between dg categories is {\em continuous} if
it preserves coproducts,
 {\em exact} if
it preserves zero objects and finite colimits, and
 {\em compact} if it preserves compact objects.

Taking ind-objects defines a faithful symmetric monoidal functor
$\Ind:\dg_\BC\to \DG_\BC$. It admits a left inverse on the subcategory of compact functors given by passing to compact objects. 
Any category $\cC \in
\dg_\BC$ is dualizable with dual the opposite category $\cC^{op} \in \dg_\BC$. Thus any category
 $\Ind\cC\in\DG_\BC$ is  dualizable with dual   the restricted opposite
category $(\Ind \cC)^\vee=\Ind (\cC^{op}) \in \DG_\BC$.

We will make heavy use of the theory of adjunctions, monads and comonads, and the Barr-Beck-Lurie theorem \cite{HA}. 
Given a monad
$T$ or comonad $T^\vee$  on a category $\cC$, we denote by 
$$\xymatrix{ \cC_T = \on{Mod}_{T}(\cC)
&
\cC^{T^\vee} = \on{Comod}_{T^\vee}(\cC)
}$$ 
the respective category of module objects
or comodule objects. 

\begin{workingdefn}[Frobenius monads] \label{frob monad}
The adjunctions appearing in this paper are {\em ambidextrous}: 
the left adjoint is canonically the right adjoint of its right adjoint  and vice versa.
Thus their compositions provide endofunctors with compatible  monadic and
 comonadic structures. We refer to  an endofunctor arising in this way as
 a {\em Frobenius monad}, though we do not independently formalize this notion in the
 $\oo$-categorical setting (see \cite{Street,Lauda} for
the notion in the discrete setting). 
A natural context for considering Frobenius monads
 is the cobordism
hypothesis with singularities \cite{jacob TFT}, where the notion of
ambidextrous adjunction captures an {oriented}
domain wall between topological field theories (as explained
pictorially by \cite{Lauda} in the discrete setting).\end{workingdefn}

\subsection{Acknowledgements} 
We would like to thank Jonathan Block and Nigel Higson for many
inspiring conversations on representations of Lie groups, and in
particular for asking what form Beilinson-Bernstein localization
should take without specified infinitesimal character. Our collaboration
with Block and Higson was funded by the SQuaRE ``The Baum-Connes
Conjecture and Geometric Representation Theory'' at the American
Institute of Mathematics, and we are indebted to AIM for its support
and hospitality. We would also like to thank Tom Nevins for helpful
discussions of localization and the paper \cite{McN}.

DBZ is partially supported by NSF grant DMS-1103525.  DN is partially
supported by NSF grant DMS-0600909.



\section{Grothendieck-Springer resolution} \label{classical}

\subsection{Recollections}
We recall here the construction of the Grothendieck-Springer
resolution of a reductive Lie algebra and the Steinberg variety.

Let $G$ be a complex reductive group. For a Borel subgroup $B\subset
G$, let $N\subset B$ denote its unipotent radical, and $H = B/N$ the
universal Cartan torus. Denote by $\fg$, $\fb$, $\fn$, and $\fh$ the
respective Lie algebras.  Let $W$ denote the Weyl group of $\fg$, and $\fc
= \fh\git W$ the affine quotient. Fix a $G$-invariant inner product on
$\fg$ to obtain an identification $\fg^*\simeq \fg$.


Let $\cB = G/B$ be the flag variety, and
 $\widetilde \cB = G/N$ the base affine space.
The natural projection $\widetilde \cB\to \cB$ is a $G$-equivariant torsor for the natural $H$-action on $\widetilde \cB$.
Such torsors correspond to homomorphisms $B\to H$, and the base affine space $\widetilde \cB\to \cB$ 
corresponds to the tautological homomorphism $B \to B/N \simeq H$.

The cotangent bundle $T^*\cB \to \cB$ classifies  pairs of a Borel subalgebra $\fb\subset \fg$
together with an element $v\in (\fg/\fb)^*\simeq  \fn$.
The moment map for the natural  $G$-action is given by the projection
$$
\xymatrix{
\mu_{\cB}:T^* \cB \ar[r] & \fg^*\simeq \fg & \mu_{\cB}(\fb, v) = v
}$$

The  cotangent bundle $T^*\widetilde \cB\to \widetilde \cB$
classifies  pairs of an element $x_\fb\in \widetilde \cB$ over a Borel subalgebra $\fb\subset \fg$
together with an element $v\in (\fg/\fn)^*\simeq  \fb$.
The moment map for the natural  $G\times H$-action is given by the projection
$$
\xymatrix{
\mu_{\widetilde \cB}:T^*\widetilde  \cB \ar[r] & \fg^* \times \fh^* \simeq \fg\times \fh & \mu_{\widetilde \cB}(x_\fb, v) = (v, [v])
}$$
where $[v]\in \fh = \fb/\fn$ denotes the image of $v\in \fb$.
 
The cotangent bundles are related by Hamiltonian reduction along the $H$-action
$$
\xymatrix{
T^*\cB = T^*(\widetilde \cB/H) \simeq  (p_\fh \circ \mu_{\widetilde \cB})^{-1}(0)/H
}$$
where $p_\fh:\fg\times\fh \to \fh$ denotes projection.

We will be interested in the quotient
$\widetilde \fg = (T^*\widetilde \cB)/H$ classifying a Borel subalgebra $\fb\subset \fg$ together with an element
$v\in (\fg/\fn)^* \simeq \fb$. The moment map for the $G$-action on $T^* \widetilde \cB$ descends to the 
Grothendieck-Springer resolution
$$
\xymatrix{
\mu_{\widetilde \fg}:\widetilde \fg\ar[r] & \fg & \mu_{\widetilde \fg}(\fb, v) = v
}
$$

The Grothendieck-Springer resolution $\mu_{\widetilde \fg}:\widetilde
\fg \to \fg$ is projective, generically finite and $G$-equivariant.
Moreover its relative dualizing sheaf is canonically trivial (and
hence the same is true of any base change of $\mu_{\widetilde \fg}$).
To see this last claim, recall we have fixed a $G$-invariant inner
product on $\fg$ to obtain a $G$-equivariant identification $\fg\simeq
\fg^*$.  This provides an isomorphism $\fg \simeq \fh \oplus \fn \oplus \fn^*$, and thus in turn induces an isomorphism of lines
$\wedge^{\dim\fg} \fg \simeq \wedge^{\dim \fh} \fh$. Thus a
trivialization of $\wedge^{\dim \fh} \fh$ trivializes the canonical
bundle of $\fg$.  Furthermore, the partial moment map $\widetilde \fg
\to \fh$ is smooth with symplectic fibers, hence
a  trivialization
of  $\wedge^{\dim \fh} \fh$ also trivializes the canonical bundle of $\widetilde\fg$.

\subsubsection{The Grothendieck-Steinberg variety}
The Grothendieck-Steinberg variety is the fiber product $\widetilde
\fg \times_\fg \widetilde \fg$ classifying triples of a pair of Borel
subalgebras $\fb_1, \fb_2 \subset \fg$ together with an element $v\in
\fb_1\cap \fb_2$. (Note here the derived fiber product coincides with
the naive fiber product.) 

It has a microlocal interpretation involving the double
coset spaces
$$\xymatrix{Z= B\bs G/B \simeq G\bs \cB \times \cB &
\widetilde Z = N\bs G/N \simeq G\bs \widetilde \cB \times \widetilde
\cB}$$ 
Namely, returning to the identification
$$
\xymatrix{
\widetilde \fg = (T^*\widetilde \cB)/H & \widetilde \cB = G/N
}$$ 
we have a similar identification
 $$ \xymatrix{ G\bs (\widetilde \fg \times_\fg \widetilde \fg) \simeq
  (T^* \widetilde Z)/H\times H}
  $$ 
  or after de-equivariantization
 $$ \xymatrix{ \widetilde \fg \times_\fg \widetilde \fg \simeq
    (pt\times_{BG}T^* \widetilde Z)/H\times H}$$ 

From this viewpoint, the fiber product in the construction of  $ \widetilde \fg \times_\fg \widetilde
\fg$ arises as the moment map equation for Hamiltonian reduction along
the diagonal $G$-action for $T^*\widetilde \cB \times T^*\widetilde
\cB$.

\subsection{Descent pattern}
Given a stack $X$, we will write $\cQ(X)$ for the symmetric monoidal dg category of quasicoherent sheaves on $X$.
All of the stacks $X$ in play will be perfect in the sense of~\cite{BFN} and so the basic structure results for $\cQ(X)$ will apply.

The natural $G$-action on $\widetilde \cB = G/N$ and induced Hamiltonian $G$-action on $\gtil$ endows
$\qc(\gtil)$ with two important compatible structures: an algebraic
action of $G$ as formalized by a $\cQ(G)$-module structure under convolution, and a $\qc(\g)$-module structure  
 via
pullback under $\mu_{\tilde \fg}:\tilde\fg\to\fg$. 
Altogether, these structures are captured by considering $\cQ(\tilde \fg/G)$ as a $\cQ(\fg/G)$-module
 via the  pullback  under the induced map $\tilde\fg/G\to\fg/G$. 
We will identify the symmetries of $\qc(\gtil/G)$
preserving this structure.  

First, the endomorphisms  of $\qc(\gtil)$ as a $\qc(\fg)$-module are given (by
\cite{BFN}) by integral transforms with kernels on the fiber product:
we have a monoidal equivalence
$$
\xymatrix{
\Phi:\qc(\widetilde \fg \times_{\fg} \widetilde \fg) \ar[r]^-\sim & \End_{\qc(\fg)}(\qc(\widetilde\fg))
& 
\Phi_{\cK}(-) = p_{2*}(\cK \otimes p_1^*(-))
}
$$
where $p_1, p_2 : \widetilde \fg \times_\fg \widetilde \fg \to \widetilde \fg$ denote the projections. In particular, the identity functor corresponds
to the integral kernel $\Delta_{\widetilde \fg*} \cO_{\widetilde \fg} \in \qc(\widetilde \fg\times_\fg \widetilde \fg)$ obtained by pushforward along the diagonal map
$$ \xymatrix{ \Delta_{\widetilde \fg} :\widetilde \fg\ar[r] &
  \widetilde \fg \times_{\fg} \widetilde\fg }
$$

Similarly  (again by
\cite{BFN}), the endomorphisms of $\qc(\gtil)$ as a $\cQ(G)$-module
are given by integral transforms with $G$-equivariant kernels on the product
$$
\xymatrix{\End_{\cQ(G)}(\qc(\gtil))\simeq \qc(G\backslash (\gtil \times \gtil)).
}$$ 

Finally  (again by
\cite{BFN}), the endomorphisms of $\qc(\gtil)$ as a Hamiltonian $G$-category, or in other words, the endomorphisms
of $\cQ(\tilde \fg/G)$ as a $\cQ(\fg/G)$-module, are given by integral transforms with equivariant kernels on the fiber product
$$
\xymatrix{
\End_{\cQ(\fg/G)} (\qc(\gtil/G))\simeq
\qc(G\backslash (\gtil \times_\fg \gtil))
}$$ 
or in other words, the monoidal dg category of equivariant
quasicoherent sheaves on the Grothendieck-Steinberg variety.

\subsubsection{Descent (co)monad}
Consider the standard adjunction and Grothendieck duality adjunction
on stable dg categories of quasi-coherent sheaves
$$
\xymatrix{
\mu_{\widetilde \fg}^*: \qc( \fg)\ar@<-0.7ex>[r] &\ar@<-0.7ex>[l] \qc(\widetilde \fg):\mu_{\widetilde \fg *}
&
\mu_{\widetilde \fg *}: \qc(\widetilde \fg)\ar@<-0.7ex>[r] &\ar@<-0.7ex>[l] \qc(\fg):\mu_{\widetilde \fg}^!
}
$$ Since the relative dualizing sheaf of $\mu_{\widetilde \fg}$ is
canonically trivial, we have a canonical equivalence $ \mu_{\widetilde
  \fg}^! \simeq \mu_{\widetilde \fg}^*, $ but we distinguish them to
avoid confusion.  By the projection formula, we can view these as
adjunctions of $\qc(\fg)$-module categories.  The adjunctions are also
evidently $G$-equivariant, and in total preserve the Hamiltonian $G$-structure.

%
%
%
%

Let $T =\mu_{\widetilde \fg }^! \mu_{\widetilde \fg *} $ denote the
resulting monad, or in other words, algebra object in the monoidal
category of linear endomorphisms $\End_{\qc(\fg)}( \qc(\widetilde
\fg))$. Likewise, let $T^\vee =\mu_{\widetilde \fg }^* \mu_{\widetilde
  \fg *} $ denote the resulting comonad. Since
$\mu^*_{\widetilde\fg}$ is conservative ($\mu_{\widetilde\fg}$ is surjective) and continuous,
$\mu^!_{\widetilde\fg}$ is as well; and since $\mu^!_{\widetilde\fg}$ is cocontinuous, 
$\mu^*_{\widetilde\fg}$ is as well.  Thus the Barr-Beck theorem
provides canonical identifications of Hamiltonian $G$-categories
$$ 
\xymatrix{  \qc(\widetilde \fg)_T \simeq
\qc(\fg) \simeq
  \qc(\widetilde \fg)^{T^\vee} 
  }
$$

By base change and standard identities, the comonad  $T^\vee$ is given by tensoring with the sheaf of functions
$\cO_{\widetilde\fg\times_\fg\widetilde \fg}= p_1^*\cO_{\tilde \fg}$ with its canonical coalgebra structure.
Likewise, the
monad $T$ is given by tensoring with the relative dualizing sheaf
$\omega_{\widetilde\fg\times_\fg\widetilde \fg/\gtil }= p_1^!\cO_{\tilde \fg}$ with its canonical algebra structure.
Note that the identification of underlying functors $T\simeq T^\vee$ reflects the equivalence
$p_1^*\simeq p_1^!$ which devolves by base change from the original
 ambidextrous adjunction of $\mu_{\tilde\fg}$.

\subsection{Specified eigenvalues} We describe here the above descent picture
to distinguished loci within $\fg$.

\subsubsection{Regular locus}
Over the open regular 
locus $\fg^{r} \subset \fg$, we have a fiber square
$$
\xymatrix{
\ar[d] \widetilde \fg^{r} = \widetilde \fg\times_{\fg} \fg^{r}\ar[r]^-{\mu_{\widetilde \fg^{r}}} & \fg^{r}\ar[d] \\
\fh \ar[r]_-\pi & \fc = \fh\git W_{\fg} }
$$

Thus descent over $\fg^{r} \subset \fg$ is simply  the base change of
descent over the geometric  invariant theory quotient.

\subsubsection{Regular semisimple locus}
Over the open regular semisimple locus $\fg^{rs} \subset \fg$,
we have a fiber square
$$
\xymatrix{
\ar[d] \widetilde \fg^{rs} = \widetilde \fg\times_{\fg} \fg^{rs}\ar[r]^-{\mu_{\widetilde \fg^{rs}}} & \fg^{rs}\ar[d] \\
\fh^{r} \ar[r]_-{\pi^r}  & \fc^{r} = \fh^{r}/ W_{\fg} }
$$
In other words, we have a free $W$-action and quotient identification
$$
\xymatrix{
W_{\fg}\times\widetilde\fg^{rs} \ar[r] & \widetilde\fg^{rs} & \fg^{rs} \simeq \widetilde\fg^{rs} /W
}
$$

Thus descent  over  $\fg^{rs} \subset \fg$ is simply equivariance for the Weyl group $W$.

\subsubsection{Nilpotent cone}

 Over the nilpotent cone $\cN = \fg \times_{\fc} \{0\} \subset \fg$,
 we have the base change
$$ \xymatrix{ \mu_{\widetilde \fg_0}: \widetilde \fg_0 = \widetilde
   \fg\times_{\fc} \{0\} \ar[r] & \cN }$$ where $\widetilde \fg_0$ is
 a non-reduced scheme with underlying reduced scheme the usual
 Springer resolution $\widetilde \cN \simeq T^* \cB$ classifying a
 Borel subalgebra $\fb\subset \fg$ together with an element $v\in
 (\fg/\fb)^* \simeq \fn$.

By construction, descent along $\mu_{\widetilde \fg_0}$ is governed by
the restricted algebra object
$$
\cO_{\widetilde \fg_0 \times_{\cN} \widetilde \fg_0} \simeq \cO_{\widetilde \fg\times_\fg \widetilde \fg}|_{\cN}
$$

\begin{remark}

To work instead with the traditional reduced Springer resolution 
$$
\xymatrix{
\mu_{\widetilde \cN}:\widetilde \cN \ar[r] &  \cN
}
$$ 
we must
 pass to ind-coherent sheaves. In applying the Barr-Beck theorem, we use that 
the adjunction
$$
\xymatrix{
\mu_{\widetilde \fg_0 *}: \qc(\widetilde \fg_0)\ar@<-0.7ex>[r] &\ar@<-0.7ex>[l] \qc(\cN):\mu_{\widetilde \fg_0}^!
}
$$ 
comprises a compact left adjoint and hence continuous right adjoint.
But in contrast, this does not hold for the adjunction
$$
\xymatrix{
\mu_{\widetilde \cN *}: \qc(\widetilde \cN)\ar@<-0.7ex>[r] &\ar@<-0.7ex>[l] \qc(\cN):\mu_{\widetilde \cN}^!
}
$$ 
For example, $\widetilde \cN$ is smooth, hence all skyscraper sheaves on it are compact, but $\cN$ is singular,
hence many skyscraper sheaves on it are not compact.
Rather we must pass to ind-coherent sheaves and work with the analogous adjunction
$$
\xymatrix{
\mu_{\widetilde \cN *}: \cQ^!(\widetilde \cN)\ar@<-0.7ex>[r] &\ar@<-0.7ex>[l] \cQ^!(\cN):\mu_{\widetilde \cN}^{!}
}
$$ 
Here by construction, the adjunction comprises a compact left adjoint and hence continuous right adjoint, 
and thus $\mu_{\widetilde \cN}^{!}$ exhibits $\cQ^!(\cN)$   as monadic over $\cQ^!(\widetilde \cN)$.
\end{remark}


\section{Beilinson-Bernstein localization}

Now we will repeat the constructions of the previous section after
quantization of the natural Poisson structures, that is, after turning on the
noncommutative deformation from cotangent bundles to $\D$-modules.

\subsection{Quantization}

Let $\fU\fg$ be the universal enveloping algebra of $\fg$, and
$\fZ\fg\subset \fU\fg$ the Harish-Chandra center.

Let $\fU\fg\module$ denote the dg category of
$\fU\fg$-modules.  Informally speaking, $\fU\fg\module$ consists of
noncommutative modules on the Poisson manifold $\fg\simeq \fg^*$.

Let $\fU\fg\perf \subset \fU\fg\module$ denote the small stable full
dg subcategory of perfect modules so that $\fU\fg\module \simeq
\Ind( \fU\fg\perf)$.

\begin{lemma}
There are canonical equivalences 
$$
\xymatrix{
\fU\fg\perf \simeq \fU\fg\perf^{op} & \fU\fg\module\simeq \fU\fg\module^\vee
}
$$
\end{lemma}

\begin{proof}
First, viewing $\fU\fg$ as a $\fU\fg$-bimodule, we define the duality identification
$$
\xymatrix{
\fU\fg\perf^{op} \ar[r]^-\sim & \fU\fg^{op}\perf & 
M\ar@{|->}[r] &  \intHom_{\fU\fg}(M, \fU\fg[\dim\fg])
}
$$

Now let $\fg^{op}$ denote the vector space $\fg$  with the opposite Lie bracket 
$$[\cdot,\cdot]_{\fg^{op}} = -[\cdot,\cdot]_{\fg}$$ The negation map $\fg\to \fg$, $v\mapsto -v$ provides
a canonical isomorphism $\fg\simeq\fg^{op}$ and hence a canonical isomorphism $\fU\fg \simeq \fU\fg^{op}$.
This establishes the first assertion, and the second then follows from the standard identity
$$
  \fU\fg\module^\vee\simeq  \Ind(\fU\fg\perf^{op})
$$
\end{proof}

\begin{remark}
The equivalence $\fU\fg\perf \simeq \fU\fg\perf^{op}$ is a twisted form of the Serre duality equivalence $\Perf(\fg^*) \simeq \Perf(\fg^*)^{op}$. Namely, the former invokes the negation on the vector space $\fg^*$ while the latter does not.
\end{remark}


Let  $\D_{\widetilde \cB} \in \qc(\widetilde \cB)$ denote the sheaf of differential operators on $\widetilde \cB$. Let $\cD_\mon(\widetilde \cB)$ denote the dg category of weakly $H$-equivariant $\D$-modules
on $\widetilde \cB$. Its objects are $H$-equivariant quasicoherent sheaves on $\widetilde \cB$ equipped with a compatible 
$H$-equivariant action of $\D_{\widetilde \cB}$.
 Informally speaking, $\cD_\mon(\widetilde \cB)$ consists of noncommutative modules on the Poisson manifold $\widetilde \fg = ( T^* \widetilde \cB)/H$.

Let $\widetilde \cD_{\cB}\in \qc( \cB)$ denote the sheaf of
$H$-invariant differential operators on $\widetilde \cB$.  
Then $\cD_\mon(\widetilde \cB)$ is equivalently the dg category of quasicoherent sheaves on $\cB$ equipped with a
compatible action of $\widetilde \cD_{\cB}$.

Let $\cD^c_\mon(\widetilde \cB) \subset\cD_\mon(\widetilde \cB)$ denote the full dg-subcategory of coherent modules so that $\cD_\mon(\widetilde \cB) \simeq \Ind( \cD^c_\mon(\widetilde \cB))$.

\begin{lemma}
Verdier duality provides canonical equivalences 
$$
\xymatrix{
\cD^c_\mon(\widetilde \cB) \simeq \cD^c_\mon(\widetilde \cB)^{op}
& 
\cD_\mon(\widetilde \cB) \simeq \cD_\mon(\widetilde \cB)^\vee
}
$$
\end{lemma}

\begin{proof}
The first assertion is Verdier duality, and the second follows from the standard identity 
$ \cD_\mon(\widetilde \cB)^\vee\simeq \Ind( \cD^c_\mon(\widetilde \cB)^{op})$.
\end{proof}

\begin{remark}
When keeping track of additional structures, it is useful to keep in mind the 
opposite  base affine space $\widetilde \cB^{op} = N\bs G$. The inverse map $G\to G$, $g\mapsto g^{-1}$
provides a canonical isomorphism $\widetilde \cB \simeq \widetilde \cB^{op}$. 
\end{remark}

Consider the localization adjunction
$$
\xymatrix{
\gamma^*: \fU\fg\module\ar@<0.7ex>[r] &\ar@<0.7ex>[l]\cD_\mon(\widetilde \cB):\gamma_*
}
$$  
$$
\xymatrix{
\gamma^*(M) = \cD_{\widetilde \cB}\otimes_{\fU\fg} M
&
\gamma_*(\cM) = \Hom(\cD_{\widetilde \cB}, \cM)
}
$$
Informally speaking, this is a quantization of the standard adjunction for the Grothendieck-Springer
resolution $\mu_{\widetilde \fg}:\widetilde\fg\to \fg$.

\begin{prop}
The right adjoint $\gamma_*$ is continuous and compact, and hence itself admits a continuous right adjoint $\gamma^!$.
Furthermore, there is a canonical identification $\gamma^! \simeq \gamma^*$.
\end{prop}

\begin{proof} 
$\gamma_*$ is continuous since its left adjoint $\gamma^*$ preserves compact objects. Moreover $\gamma_*$ itself preserves compact objects, indeed it's clear that coherent $\cD$-modules get taken to finitely presented $\fU\fg$-modules.
We also have a compatibility with duality of the form $$
\xymatrix{
\ar@<-5ex>[d]_-{\gamma_*} \cD^c_\mon(\widetilde \cB) \simeq \cD^c_\mon(\widetilde \cB)^{op}\ar@<5ex>[d]^-{\gamma_*^{op}}\\
\fU\fg\perf \simeq \fU\fg\perf^{op}
}
$$

To check the equivalence $\gamma^!\simeq \gamma^*$, we apply the Rees construction. Namely, we have a $\Gm$-equivariant $\mathbb A^1$-family of adjunctions $(\gamma^*, \gamma_*, \gamma^!)$ degenerating to the classical version of the localization functors, associated to the $\Gm$-equivariant Poisson map which is the Grothendieck-Springer resolution  $\mu_{\widetilde \fg}:\widetilde\fg\to \fg$. We need to show the canonical map $\gamma^* \to \gamma^!$ is an equivalence, i.e., its cone vanishes. It's enough to check this when applied to the Rees construction of $\fU\fg$ itself. At $\hbar=0$, the cone vanishes by the Calabi-Yau property of $\mu_{\widetilde \fg}$. Hence it vanishes for all $\hbar$ by semicontinuity.  
\end{proof}

\subsubsection{Linearity} 
Let $\fZ\fg  \subset \fU\fg$ be the Harish-Chandra center,
and $\fU\fh$ the universal enveloping algebra of $\fh$.
We also have the canonical embedding
$\fZ\fg\subset \fU\fh$ as the $\rho$-shifted Weyl invariants.

Observe that $\fU\fg\module$ is naturally $\fZ\fg$-linear, and
 $\cD_\mon(\widetilde \cB)$ is naturally $\fU\fh$-linear and hence $\fZ\fg$-linear.

\begin{lemma}
The adjunctions 
$$
\xymatrix{
\gamma^*: \fU\fg\module\ar@<0.7ex>[r] &\ar@<0.7ex>[l]\cD_\mon(\widetilde \cB):\gamma_*
&
\gamma_*: \cD_\mon(\widetilde \cB)\ar@<0.7ex>[r] &\ar@<0.7ex>[l]\fU\fg\module:\gamma^!
}
$$ 
are naturally $\fZ\fg$-linear.
\end{lemma}

\begin{proof}
This follows from the compatibility between the action $\fU\fh\to \Gamma(\wt{\cD_\cB})$ and the Harish-Chandra homomorphism $\fZ\fg\hookrightarrow \fU\fh$,
see~\cite[Lemma 3.1.5]{BMR1}.
\end{proof}


\subsection{Beilinson-Bernstein localization and the nil-Hecke algebra}\label{nil section}
In this section we describe how one can easily apply Barr-Beck-Lurie to identify representations of $\fg$ with modules over the nil-Hecke algebra inside globally generated monodromic $\cD$-modules on $\wt{\cB}$. This realization won't be used in what follows, where we apply Barr-Beck-Lurie in the opposite direction.

Let $$\wt{\fU\fg}=\fU\fg \otimes_{\fZ\fg} \fU\fh$$ denote the extended enveloping algebra. 

We will use the following proposition of Mili\v ci\'c describing monodromic differential operators on the basic affine space:
\begin{prop}\label{milicic}~\cite[Lemma 3.1]{milicic1},~\cite[Theorem C.6.5]{milicic2}
There is an equivalence 
$$
\gamma_*(\cD_{\widetilde \cB} ) \simeq \wt{\fU\fg}
$$ 
i.e. we have an isomorphism of global sections
$$\Gamma(G/N,\cD_{\widetilde\cB})^H=\Gamma(G/B, \wt{\cD})\simeq \wt{\fU\fg}$$
and vanishing of higher cohomologies
$$R^i\Gamma(G/B, \wt{\cD})=0 \hskip.3in (i>0).$$
\end{prop}

\begin{prop} \label{global sections BB}
Localization and global sections induce an equivalence
$$\cD_\mon(\widetilde \cB)^{glob}\simeq \wt{\fU\fg}\module,$$
where the left hand side is the category of {\em globally generated} monodromic $\cD$-modules, i.e., the category generated by the sheaf of differential operators.
\end{prop}

\begin{proof}
The functor $\gamma_*$ has both left and right adjoints. Thus to apply the Barr-Beck-Lurie theorem (in the monadic form) we just need to ensure conservativity. This can be always be achieved formally in the setting of stable categories by killing the objects that are sent to zero by the given functor, i.e., passing to the left orthogonal of the subcategory of $\gamma_*$-null objects. Equivalently, we pass to the category generated by the image of the left adjoint $\gamma^*$ -- in other words the category of globally generated objects. Finally the monad $\gamma_*\gamma^*$ itself is identified by Proposition~\label{milicic} with the extended enveloping algebra, as desired.
\end{proof}

The nil-Hecke algebra of Kostant and Kumar $\BH$ is the subalgebra of endomorphisms of $\C[\fh^\ast]$ generated by $Sym(\fh)=\C[\fh^\ast]$ and the Demazure, or divided-difference, operators associated to simple reflections $\sigma_i\in W$ with associated simple roots $\alpha_i$
$$A_i=(1-\sigma_i)/\alpha_i.$$ 
It has the interpretation as the algebra which controls descent along the map $$\fz:\fh^\ast\to \fh^\ast//W$$ from the dual Cartan to its coarse quotient by the Weyl group: namely, there is an equivalence of categories $\cQ(\fh^\ast//W)\simeq \cQ(\fh^\ast)^{\BH}$ (see e.g.~\cite{Lonergan}, which proves the corresponding statement for the coarse quotient of the reflection representation by any Coxeter group).
Hence applying the descent to $\fU\fg$-modules we can identify them with $\BH$-modules in $\wt{\fU\fg}$-modules, and hence from Proposition~\ref{global sections BB} we find the following:

\begin{corollary} There is an equivalence of categories $$(\cD_\mon(\widetilde\cB)^{glob})^\BH\simeq \fU\fg\module$$
where the nil-Hecke algebra acts on the left hand side through its $\fU\fh$-linearity.
\end{corollary}

\subsection{Symmetries} 
We next introduce  quantum analogues of the previously encountered Hamiltonian $G$-actions.
Following \cite{BD,FG},
by a
 {\em de Rham
  $G$-category}, we mean a dg category with an algebraic $G$-action  that is   
 infinitesimally trivialized, or in other words, the induced action of the formal group of $G$ is trivialized.
This can be formalized by saying that a dg category is a 
module for the monoidal
dg category $\D(G)$ of $\D$-modules on $G$ under  convolution.

The primary examples are the dg category $\D(X)$ of $\D$-modules on a $G$-variety
$X$, and the dg category $\fU\fg\module$ with its conjugation $G$-action. 
These can be unified by  considering more generally $\D(G)$-modules of the form $\D_{G'\text{\it -mon}}(X)$
where  $X$ is a  $G\times G'$-variety, and we take weakly $G'$-equivariant $\D$-modules. In particular, for $G = X = G'$ with $G$ acting on the left and $G'$ on the right, we have the $\D(G)$-linear equivalence $\fU\fg\module \simeq \D_{G'\text{\it -mon}}(G)$. (Since the action is free and transitive, we can trivialize the underlying quasi-coherent sheaf of a weakly equivariant $\D$-module, and then realize invariant sections as a module over invariant vector fields.)
Informally speaking, this is a quantum analogue of the identification $\fg^*\simeq (T^*G)/G$.
Similarly, $\cD_\mon(\widetilde \cB)$ comes equipped with a natural $\D(G)$-module structure.

\begin{remark} Within the stable setting, there are two equivalent dual  formulations of a de Rham $G$-category.
By definition, we have taken de Rham $G$-category to mean a module for the monoidal dg category $\D(G)$
of $\D$-modules on $G$ under convolution. But one can observe that for $X$ a smooth variety, $\D(X)$ is dualizable as a plain dg category. Furthermore, 
it is self-dual so that
for maps $f:X\to Y$ of smooth varieties, 
pullback is dual to pushforward (by the projection formula). Thus a de Rham $G$-category could equivalently
be taken to mean a comodule for $\D(G)$ equipped with its coconvolution coalgebra structure.
\end{remark}

 The following is evident from the constructions.

\begin{lemma}
The adjunctions 
$$
\xymatrix{
\gamma^*: \fU\fg\module\ar@<0.7ex>[r] &\ar@<0.7ex>[l]\cD_\mon(\widetilde \cB):\gamma_*
&
\gamma_*: \cD_\mon(\widetilde \cB)\ar@<0.7ex>[r] &\ar@<0.7ex>[l]\fU\fg\module:\gamma^!
}
$$ 
are naturally $\D(G)$-linear.
\end{lemma}

Consider the stack $\widetilde Z = N\bs G/N$.
Let $\D_{\bimon}(\widetilde Z)$ denote the dg category of
$H\times H$-weakly equivariant $\D$-modules on $\widetilde Z$.
 Informally speaking, $\D_{\bimon}(\widetilde Z)$ consists of noncommutative modules on the  Grothendieck-Steinberg stack $G\bs (\widetilde \fg \times_{\fg} \widetilde\fg)$.

Convolution equips $\D_{\bimon}(\widetilde Z)$ with a natural
monoidal structure, and $\D_\mon(\widetilde \cB)$ with a natural right
$\D_{\bimon}(\widetilde Z)$-module structure commuting with its
natural left $\D(G)$-module structure. This is a quantum analogue
of the convolution pattern for  sheaves on $\gtil\times_{\fg} \widetilde \fg$ acting by integral transforms on sheaves on $\gtil$ respecting the Hamiltonian
$G$-structure. The following is  the quantum analogue of the  result quoted from \cite{BFN} earlier that such integral transforms
are precisely the symmetries respecting  the Hamiltonian
$G$-structure.

\begin{thm} Convolution  provides a monoidal  equivalence
$$ \xymatrix{ \Phi: \D_{\bimon}(\widetilde Z)\ar[r]^-\sim
    & \End_{\D(G)}(\cD_\mon(\widetilde \cB)) }
$$
\end{thm}

\begin{proof}
Let us begin by forgetting the $\D(G)$-module structure of $\cD_\mon(\widetilde \cB)$. Then by  a monodromic version of~\cite[Theorem 1.14]{character}, we have a monoidal equivalence
$$ \xymatrix{ \Phi': \D_{\bimon}(\widetilde \cB \times\widetilde\cB)\ar[r]^-\sim
    & \End(\cD_\mon(\widetilde \cB)) 
    & 
\Phi'({\cK})(-) = p_{2*}(\cK \otimes p_1^*(-))
}
$$
where $p_1, p_2 : \widetilde \cB \times\widetilde\cB \to \widetilde \cB$ denote the projections. 

Returning $\D(G)$-module structures
to the picture, $\Phi'$ is evidently $\D(G)$-linear by standard identities.
Moreover, $\D(G)$-linear endomorphisms of $\cD_\mon(\widetilde \cB)$ are simply the invariants
$$ \xymatrix{ 
\End_{\D(G)}(\cD_\mon(\widetilde \cB))  = \Hom_{\D(G)}(\D(pt), \End(\cD_\mon(\widetilde \cB)) 
}
$$
By descent along $pt\to BG$, the invariants can be calculated as comodules 
$$ \xymatrix{ 
 \Hom_{\D(G)}(\D(pt), \End(\cD_\mon(\widetilde \cB)) \simeq 
 \End(\cD_\mon(\widetilde \cB))^{\cO_G}
}
$$
for the canonical coalgebra $\cO_G\in \D(G)$ given by the structure sheaf.

On the other hand, by another application of descent, $\cO_G$-comodules in  $\D_{\bimon}(\widetilde \cB \times\widetilde\cB)$ are precisely $G$-equivariant bimonodromic $\D$-modules on $\widetilde \cB \times\widetilde\cB$. Now the theorem follows from the identification  $\widetilde Z \simeq G\bs (\widetilde \cB \times\widetilde\cB)$.
\end{proof}






\subsection{Universal Weyl sheaf}
Recall the  adunctions
$$
\xymatrix{
\gamma^*: \fU\fg\module\ar@<0.7ex>[r] &\ar@<0.7ex>[l]\cD_\mon(\widetilde \cB):\gamma_*
&
\gamma_*: \cD_\mon(\widetilde \cB)\ar@<0.7ex>[r] &\ar@<0.7ex>[l]\fU\fg\module:\gamma^!
}
$$ 
with ambidextrous identification $\gamma^!\simeq\gamma^*$.

Let $T =\gamma^! \gamma_* $ denote the resulting monad, or in other words, algebra object in the monoidal category of endomorphisms $ \End_{\D(G)}(\cD_\mon(\widetilde \cB))$.  
Likewise, let $T^\vee =\gamma^* \gamma_*$ denote the resulting comonad.

 \begin{lemma}
 $\gamma^!$  is conservative, and hence $\gamma^*$ is as well.
 \end{lemma}
 
 \begin{proof}
From Proposition~\ref{milicic}
we find
$$
\Hom(\cD_{\widetilde \cB},  \gamma^!(M )) \simeq \Hom( \fU\fg \otimes_{\fZ\fg} \fU\fh, M)
$$
Since $\fU\fh$ is free (of rank $|W|$) over $\fZ\fg$, we conclude that this Hom space vanishes if and only if $M\simeq 0$, hence $M \not \simeq 0$ implies $\gamma^!(M) \not \simeq 0$.
 \end{proof}
 
Since $\gamma^!$ is continuous and conservative and 
 $\gamma^*$ is cocontinuous and conservative,  the Barr-Beck theorem provides canonical identifications
$$
\xymatrix{
 \Mod_T (\cD_\mon(\widetilde \cB)) \simeq \fU\fg\module \simeq \on{Comod}_{T^\vee} (\cD_\mon(\widetilde \cB))
}
$$

We would like to explicitly describe the integral kernel giving rise to $T \simeq T^\vee$ under the equivalence
$$
\xymatrix{
\Phi: \D_{\bimon}(\widetilde Z)\ar[r]^-\sim &  \End_{\D(G)}(\cD_\mon(\widetilde \cB))
}
$$

\begin{defn}
The {\em universal Weyl sheaf} $\cW \in \D_{\bimon}(\widetilde Z)$ is the sheaf of differential operators on $\widetilde Z$
with its canonical $H\times H$-weakly equivariant structure.
\end{defn}

\begin{remark}  Let us spell out this definition. By quantum Hamiltonian reduction (and under the
identification $\widetilde Z \simeq G\bs (G/N \times G/N)$), the
pullback of $\cW$ along the natural quotient map
$$
\xymatrix{
r: G/N \times G/N  \ar[r] & N\bs G/N
}$$ 
 is the $G$-strongly equivariant 
$\D$-module
$$
r^*\cW = \D_{G/N \times G/N}/(\fg)
$$ where $(\fg) \subset \D_{G/N \times G/N}$ is the left ideal
generated by vector fields arising from the diagonal $G$-action on
$G/N \times G/N$.  Thus $\cW$ is the quantum analogue of the structure
sheaf of $G\bs (\gtil\times_\fg\gtil)$.
\end{remark}


\begin{remark}
One can write down an explicit Frobenius algebra structure on $\cW$ but its construction is an explicit unwinding of that given by the adjunctions of Theorem~\ref{thm weyl kernel} below.
\end{remark}

In parallel with the commutative case, the identification of $\cW$ as the
 integral kernel giving rise to $T \simeq T^\vee$
 can be viewed as a microlocal version of $G$-equivariant base change along the diagram
$$ \xymatrix{ &
   \ar[dl]^-{p_{1*}}\D_{H\times H}(\widetilde \cB \times \widetilde
  \cB) & \\  \D_\mon(\widetilde \cB)
  && \ar[ul]^-{p_2^*}\D_\mon(\widetilde \cB) \ar[dl]_-{\gamma_*} \\ & 
\ar[ul]_-{\gamma^*} \fU\fg\module & }
$$
This is formalized in the argument for the following assertion.

\begin{thm}\label{thm weyl kernel} The monoidal equivalence
$$
\xymatrix{
\Phi:\D_{\bimon}(\widetilde Z)\ar[r]^-\sim &  \End_{\D(G)}(\cD_\mon(\widetilde \cB))
}
$$ takes the universal Weyl sheaf $\cW$ to the endofunctor $T\simeq T^\vee$. Thus $\cW$ inherits the structures of algebra
and coalgebra in
$\D_{\bimon}(\widetilde Z)$, and we have equivalences
$$\ \D_\mon(\widetilde \cB)_{\cW}\simeq \fU\fg\module\simeq \D_\mon(\widetilde \cB)^{\cW}.$$
\end{thm}

\begin{proof}
On the one hand, given an object $\cM\in \D_\mon(\widetilde \cB)$, we
find
$$
T(\cM) \simeq \gamma^*\gamma_*\cM \simeq \widetilde \cD_{\cB} \otimes_{\fU\fg} \Hom_{\widetilde \cD_{\cB}}(\widetilde \cD_{\cB}, \cM)
$$ Thus it is the $\fU\fg$-coinvariants of the intermediate functor
$$
\xymatrix{
T'(\cM) =  \widetilde \cD_{\cB} \otimes \Hom_{\widetilde \cD_{\cB}}(\widetilde \cD_{\cB}, \cM)
}$$ 
Moreover, under the identification
$$ \xymatrix{ \Phi': \D_{\bimon}(\widetilde \cB \times\widetilde\cB)\ar[r]^-\sim
    & \End(\cD_\mon(\widetilde \cB)) 
}
$$
observe that $T'$ corresponds to the integral kernel $\widetilde \D_{\cB} \boxtimes {\mathbb D}(\widetilde \D_{\cB})$. (Indeed the integral kernel $M\boxtimes N$ represents the functor $N\otimes \Hom({\mathbb D}M,-)$). Moreover, the Calabi-Yau structure on $G/N$ allows us to identify $\widetilde \D_{\cB}$ with its dual (with suitable $H$-monodromic structure).
On the other hand, as discussed above, $\cW$ is simply the
$\fU\fg$-coinvariants of  $\widetilde \D_{\cB}
\boxtimes \widetilde \D_{\cB}$. Thus since all functors are
continuous, taking $\fU\fg$-coinvariants can be equivalently performed
on the integral kernel or on the result of the integral transform.

All of the above arguments are manifestly $G$-equivariant and so descend to give the assertion. 
\end{proof}


\subsection{Specified infinitesimal character}\label{specified}

Recall that the commutative algebra $\fU\fh \otimes \fU\fh = \cO( \fh^*\times \fh^*)$ acts by central endomorphisms on $\D_{\bimon}(\widetilde Z)$.
The action factors through the closed subscheme $\Gamma\subset \fh^* \times \fh^*$ given by the union of the graphs of Weyl group elements
$$
\xymatrix{
\Gamma = \coprod_{w\in W} \Gamma_w & \Gamma_w = \{(\lambda, w\lambda) \in \fh^*\times \fh^*\}
}
$$

To better understand the universal Weyl sheaf  $\cW \in \D_{\bimon}(\widetilde Z)$,
let us restrict one of its monodromies and calculate its resulting fiber.
The composite projection  to either factor 
$$
\xymatrix{
\Gamma \ar@{^(->}[r] & \fh^*\times \fh^* \ar[r] & \fh^*
}
$$
is a finite flat map. All of what follows is symmetric in the two projections, so let us focus on the projection to the second factor.

First, let us identify the fibers of $ \D_{\bimon}(\widetilde Z)$ along the projection to the second factor. 
For simplicity, let us also forget the $H$-weak equivariance along the first factor,  i.e. consider $\D$-modules on $N\bs G/N$ which are $H$-weakly equivariant on the right. The fiber of this category over $\lambda \in \fh^*$ is canonically equivalent to
the dg category $\D_\lambda(N\bs \cB)$ of $N$-strongly equivariant $\lambda$-twisted $\D$-modules on the flag variety
$\cB$.
This in turn is the full subcategory of those
$\lambda$-twisted $\D$-modules on $\cB$ that are locally constant along Schubert cells. To recover the corresponding fiber of $ \D_{\bimon}(\widetilde Z)$ we should then reimpose weak $H$-equivariance on the right.

Now let us identify the corresponding fiber 
$$\cW_\lambda \in \D_\lambda(N\bs \cB)
$$ of the universal Weyl sheaf. This is a regular holonomic  $\lambda$-twisted $\D$-module on the flag variety $\cB$ locally constant along Schubert cells.

For concreteness, we will consider several specific cases: (1) $\lambda$ generic (regular and not at all integral), (2) $\lambda$ regular and integral,  and (3) $\lambda= 0$ trivial.

(1) When $\lambda$ is generic, we have a direct sum decomposition
$$
\xymatrix{
\cW_\lambda \simeq \oplus_{w\in W} \cW_{\lambda, w\lambda}
}$$
Each summand admits the description as a standard or equivalently costandard extension off of a Schubert cell
$$
\xymatrix{
\cW_{\lambda, w\lambda} \simeq j_{w!} \cO_{\lambda, w} \simeq  j_{w*} \cO_{\lambda, w}
}$$
where $j_{w}:\cB_w \to \cB$ is the inclusion of the $w$-Schubert cell for $w\in W$, and $\cO_{\lambda, w}$ is the 
$\lambda$-twisted structure sheaf of $\cB_w$.

(2) Suppose $\lambda$ is regular and integral, and let $w_0 \in W$ be the Weyl group element such that $\lambda = w_0 \lambda_0$ for  dominant $\lambda_0$.

We can tensor by the line bundle $\cO(\lambda)$ to obtain an identification
$$
\xymatrix{
\D( \cB) \ar[r]^-\sim & \D_\lambda(\cB) 
&
M\ar@{|->}[r] & M\otimes \cO(\lambda)
}$$
This is convenient since we will describe $\cW_\lambda$ in terms of the natural monoidal structure on the dg category $\D(B\bs \cB)$ of $B$-equivariant $\D$-modules on the flag variety $\cB$. Namely, under the above identification, we have a direct sum 
decomposition
$$
\xymatrix{
\cW_\lambda \simeq \oplus_{w\in W} \cW_{\lambda, w\lambda}
}$$
Each summand admits the description as the convolution of standard extensions $T_{w*} = j_{w*}\cO_w$
and costandard extensions  $T_{w!} = j_{w!}\cO_w$ off of Schubert cells 
$$
\xymatrix{
\cW_{\lambda, w\lambda}  \simeq T_{w_0!} * T_{w w_0^{-1} *} \simeq T_{w_0 w w_0^{-1} *}
}$$

(3) When $\lambda = 0$ is trivial, the fiber $\cW_{0}$ is the maximal tilting sheaf on $\cB$. Namely, within $\D(N\bs \cB)$,  it is the projective cover of the skyscraper sheaf at the closed Schubert cell.


\subsection{Example: $SL_2$}\label{specified}

We offer here a brief discussion of the structure of the universal Weyl sheaf $\cW$ in the case when $G=SL_2$.
Already here one can see how intricate topology is packaged in the simple algebra of $\cW$. More specifically,
the simple notion of differential operators on $N\bs SL_2/N$ interpolates between standard, costandard and tilting sheaves
as we specialize parameters.

Let us identify $\fh^*\simeq \BA^1$ so that $\Lambda^* \simeq \BZ$ with the reflection action of the Weyl group $W \simeq \BZ/2$
centered at $-\rho = -1\in  \BZ\subset\BA^1$. Thus $-1\in \BZ\subset \BA^1$ is the unique singular parameter, the rest of the integers $\BZ\setminus \{-1\}$ are regular integral, and the rest of the parameters $\fh^*\setminus \BZ$ are generic (not at all integral). 

Let us identify 
$$
\xymatrix{
SL_2/N \simeq \BA^2 \setminus \{(0,0)\} = \Spec \BC[x,y] \setminus\{(0,0)\}
}$$ so that 
 the left action of $B$ has orbits 
$$\xymatrix{
i:V = \G_m \times \{0\} \ar@{^(->}[r] & SL_2/N
}$$
and the complement 
$$
\xymatrix{
j:U = \BA^1 \times \G_m  \ar@{^(->}[r] & SL_2/N
}
$$ 
and the right action of $H \simeq \G_m$ is the usual scaling dilation.
Thus the left action of the diagonal torus $T \subset B$
coincides with the right action of $H$ on the closed orbit $V $ and is its inverse on the open orbit 
$U$. Note that the left action of $N\simeq \BA^1$ consists of the individual points of 
the closed orbit $V $ and the slices
$ \BA^1 \times \{y\} \subset U $ of the open orbit.

The $N$-equivariant ring of differential operators on $G/N$ is given by Hamiltonian reduction
$$
\cW = \cD_{G/N}/(\fn)
$$
where $(\fn) \subset \cD_{G/N}$ is the left ideal generated by the vector field $y\partial_x$ arising from the left $N$-action.
Thus the coisotropic (but not Lagrangian) singular support of $\cW$ is the union of conormals to $N$-orbits and explicitly cut out by the single equation $y\xi_x = 0$.

Let us specialize $\cW$ to prescribed monodromies for the right $H$-action. For $\lambda\in \fh^* \simeq \BA^1$, the corresponding fiber is the quotient
$$
\cW_\lambda = \cD_{G/N}/(\fn, \fh_\lambda)
$$
where $(\fn,  \fh_\lambda) \subset \cD_{G/N}$ is the left ideal generated by the vector field $y\partial_x$ arising from the left $N$-action, and the differential operator $x\partial_x + y\partial_y - \lambda$ coming from prescribing the monodromy of the right $H$-action.
Thus $\cW_\lambda$ is regular holonomic with singular support the union of conormals to $B$-orbits and explicitly cut out by the equations $y\xi_x = x\xi_x + y\xi_y = 0$. From these equations, we see its characteristic cycle is the weighted sum of conormals
to $B$-orbits
$$
cc(\cW_\lambda) = 2\cdot T^*_{V} + T^*_{U} 
$$

Now when $\lambda\in \fh^*\setminus \BZ$ is generic, $\cW_\lambda$ splits as a direct sum
$$
\cW_\lambda = i_*\cL_{V, \lambda} \oplus j_*\cL_{U, \lambda}
$$
of the standard extensions of the local systems on $B$-orbits
$$
\xymatrix{
\cL_{V, \lambda} = \D_V/(x\partial_x - \lambda) & 
\cL_{U, \lambda} = \D_U/(\partial_x, y\partial_y - \lambda) 
}$$
Observe that because $\lambda$ is generic, the standard extension off of the open orbit $U$ is equivalent to the costandard extension
$$
j_!\cL_{U, \lambda} \simeq j_*\cL_{U, \lambda}
$$
(Of course, there is no difference in the standard and costandard extensions off of the closed orbit $V$.)
The characteristic cycles of the summands are the sums of conormals to $B$-orbits
$$
\xymatrix{
cc(i_*\cL_{V, \lambda}) = T^*_{V}  & 
cc(j_*\cL_{U, \lambda}) = T^*_{V}  + T^*_U
}$$

When $\lambda\in\BZ\setminus\{-1\}$ becomes regular integral, $\cW_\lambda$ continues to split as a direct sum
but now in a more delicate form. Namely, when $\lambda \in \{0, 1,2,\ldots\}$ is ``positive", we have
$$
\cW_\lambda = i_*\cL_{V, \lambda} \oplus j_*\cL_{U, \lambda}
$$
and when $\lambda \in \{-2,-3,-4,\ldots\}$ is ``negative", we have
$$
\cW_\lambda = i_*\cL_{V, \lambda} \oplus j_!\cL_{U, \lambda}
$$
From the perspective of the monadic symmetries governing localization, this asymmetry reflects the choice of global sections functor. For example, going back to Borel-Weil-Bott, we see  that the global sections of line bundles leads to 
the asymmetry of cohomological shifts.

Finally, when $\lambda = -1$ is singular, $\cW_0$ no longer splits as a direct sum but becomes
the indecomposable tilting extension of the structure sheaf of the open orbit
$$
\cW_0 = \cO_U^{\text{\it tilt}} =\D_{G/N}/(y\partial_x, x\partial_x + y\partial_y)
$$
Thus it is self-dual and admits an increasing filtration
$$
\cW^0_0 \subset \cW^1_0 \subset \cW_0
$$
with associated graded
$$
\xymatrix{
\cW^0_0 \simeq i_*\cO_V & 
\cW^1_0/\cW^0_0 \simeq \cO_{G/N} &
\cW_0/\cW^1_0 \simeq i_*\cO_{V}
}
$$

To see any of the preceding identifications of $\cW_\lambda$ explicitly, one can restrict to the transverse line $\BA^1 \simeq \{x=1\}$
to find
$$
\cW_\lambda|_{\{x=1\}} \simeq \D_{\BA^1}/(y^2\partial_y - \lambda y) \simeq  \D_{\BA^1}/(y\partial_y y - (\lambda+1) y)  
$$
In particular, when $\lambda = -1$ is singular, we find the traditional algebraic presentation $ \D_{\BA^1}/(y\partial_y y)$ 
of the tilting extension of the structure sheaf of $\G_m$ to all of $\BA^1$.
%
%
%
%
%
%
%


\section{Appendix: Demazure Descent}\label{Demazure section}
In this section we present a simple uniform description of the Grothendieck-Springer monad on $\cQ(\gtil)$ expressing descent to $\fg^\ast$ and the Weyl monad on $\cD(G/N)_H$ expressing localization of $U\fg\module$. Namely we show they both come from the {\em Demazure monad} $\fd$, an algebra in the coherent or Demazure Hecke category $\DD=(\cQ^!(\BGB),\ast)$, via lax actions of $\DD$ on the corresponding categories. We are grateful to Arkhipov and Kanstrup~\cite{AK1} for suggesting the relevance of Demazure algebras to the results of this paper.

In order to see the Grothendieck-Springer resolution and Beilinson-Bernstein localization on equal footing, it is convenient to pass from sheaves on cotangents and $\cD$-modules to the dual perspective of sheaves on Dolbeault and de Rham spaces, respectively. The de Rham functor of a smooth scheme $X$ can be described as the quotient of $X$ by the formal neighborhood of the diagonal, while the Dolbeault functor is the relative classifying stack of the formal group of the tangent bundle of $X$. The general formalism of inf-schemes~\cite{GR} handles such objects (more generally, quotients of schemes by formal groupoids) on an equal footing with ordinary schemes. The extension of the functor $\cQ^!$ of ind-coherent sheaves to (correspondences of) inf-schemes developed in~\cite{GR} allows one to treat $\cD$-modules (together with strong or weak equivariance) and coherent sheaves on an equal footing.

\begin{notation} In this appendix, for a smooth scheme $X$ we denote by $\bX$ either 
\begin{enumerate}
\item the de Rham space $X_{dR}$, in which case $\cQ^!(\bX)=\cD(X)$,
\item the Dolbeault space $X_{Dol}$, i.e., the classifying space of the formal group of the tangent bundle of $X$, in which case
$\cQ^!(\bX)=\cQ^!(T^*X)$.
\end{enumerate} 
\end{notation}

(We could likewise work with the Hodge inf-scheme $X_{Hod}\to{\mathbb A}^1/\Gm$ interpolating $X_{dR}$ and $X_{Dol}$~\cite{simpson}, with sheaves given by modules for the Rees algebra of $\cD$.)

\begin{example}
\begin{enumerate}
\item For a $G$-space $X$, the categories $\cQ^!(\bX/G)$ and $\cQ^!(\bX/\bG)$ recover the categories of weakly and strongly $G$-equivariant $\cD$-modules on $X$, respectively. The pullback functor $p^!$ along $p:\bX/G\to \bX/\bG$ is identified with the standard functor forgetting from strong to weak equivariance.
\item The category $\cQ^!(\bG/G)$ is identified with $\cQ(T^* G)^G\simeq \cQ(\fg^\ast)$ in the Dolbeault setting and $\cD(G)_G\simeq U\fg\module$ in the de Rham setting.
\item The category $\cQ^!(\bG/\bN H)$ is identified with $\cQ((T^* G/N)/H)=\QC(\gtil)$ (Dolbeault) and $\cD(G/N)_H$ (de Rham).
\end{enumerate}
\end{example}

The two spaces above are related by a correspondence
\xymatrix{&\bG/B\ar[dl]_-p\ar[dr]^-q&\\
\bG/\bN H && \bG/G }

\begin{lemma}\label{identify functors}
The functor $q_*p^!:\cQ^!(\bG/\bN H)\to\cQ^!(\bG/G)$
is identified with pushforward $\pi_*$ along the Grothendieck-Springer resolution 
$\pi:\gtil\to \fg^\ast$ in the Dolbeault setting and with global sections $\gamma_*:\cD(G/N)_H\to U\fg\module$  in the de Rham setting.
\end{lemma}

It is now easy to express the resulting monads geometrically by composing this correspondence with its opposite:
\begin{equation}\label{BGB correspondence}
\xymatrix{&&\ar[dl]\bG/B\times_{pt/B}  \BGB \ar[dr]&&\\
&\bG/B\ar[dl]\ar[dr]&&\bG/B\ar[dl]\ar[dr]&\\
\bG/\bN H && \bG/G && \bG/\bN H}
\end{equation}
We see that the Grothendieck-Springer and Beilinson-Bernstein adjunctions are controlled by the groupoid $\BGB$ through its action on $\bG/B$. We will now make this relation precise.

\subsection{Equivariance for groupoids}

In order to gain some perspective on the Demazure monad, it is useful to consider it in the setting of groupoid actions.
Thus for an ind-proper morphism $p:X\to Y$ let $\cG=X\times_Y X$ denote the descent groupoid, with source and target the two projections
$\pi_1,\pi_2:\cG\to X$. (We will apply this to $p:X=pt/B\to Y=pt/G$, with $\cG=\BGB$.)

Let $\cH=\cQ^!(\cG)$, the Hecke category of $\cG$, 
denote the monoidal category of ind-coherent sheaves under convolution.
It is inherited on applying $\cQ^!$ to the structure on $\cG$ of algebra object in correspondences.
(See~\cite[Sections II.2.5.1, III.3.6.3]{GR} for a general treatment of ind-proper groupoids and more generally monoid or Segal stacks and the corresponding convolution categories.)  The diagonal embedding (unit map) $i:X\to \cG$ induces a monoidal functor $\cQ^!(X)\to \cH.$

Let ${\mathbf H}$, the Hecke algebra of $\cG$, denote the descent monad for $p:X\to Y$, an algebra object structure on the functor $\pi_{2,*}\pi_1^!\simeq p^!p_*\in End(\cQ^!(X))$. Thus by ind-proper descent we have an identification $$
{\mathbf H}\module_{\cQ^!(X)}\simeq \cQ^!(X)^\cG\simeq  \cQ^!(Y)$$ of ${\mathbf H}$-modules, $\cG$-equivariant sheaves on $X$ and sheaves on $Y$.

\begin{defn} For a $\cH=\cQ^!(\cG)$-module $\cM$, we define the category of $\cG$-equivariant objects to be $$\cM^\cG:=Hom_{\cH}(\cQ^!(X),\cM).$$
\end{defn}

The following proposition (while not strictly needed) gives a handy picture of equivariance as modules for a 
categorical form of the averaging idempotent in a finite group algebra:

\begin{prop} 
\begin{enumerate}
\item The object $e=\omega_{\cG}\in \cH$ has a canonical structure of algebra object, which is taken to the Hecke algebra ${\mathbf H}$ under the monoidal action map $\cH\to End(\cQ^!(X))$. Thus we have an equivalence $\cQ^!(Y)=\cQ^!(X)^\cG\simeq \cQ^!(X)_{e}$ of equivariant sheaves with $e$-modules in $\cQ^!(X)$.
\item More generally, for any $\cH$-module $\cM$ we have an identification $\cM^\cH\simeq \cM_e$ of equivariant objects with $e$-modules in $\cM$.
\end{enumerate}
\end{prop}

\begin{proof}
The first assertion follows from base-change and ind-proper descent.
It follows from the ind-properness of $p$ that $\cH$ is rigid, while the QCA property of $X$ implies the self-duality of $\cQ^!(X)$ over $k$. Together these properties imply that $\cQ^!(X)$ is naturally self-dual over $\cH$.
It follows that for any $\cH$-module $\cM$ we have 
\begin{eqnarray*}
\cM^\cG&=&Hom_\cH(\cQ^!(X),\cM)\\
&\simeq& \cQ^!(X)\ot_\cH \cM \\
&\simeq& \cH_e \ot_\cH\cM\\
&\simeq& \cM_e
\end{eqnarray*}
where the last step is~\cite{BFN}[Proposition 4.1], or rather its 
straightforward extension from symmetric monoidal to monoidal categories from~\cite{morita}[Proposition 3.1]. (In particular we have $\cQ^!(X)^\cG=\cQ^!(Y)$ so the notation is consistent.)
\end{proof}



\subsection{The Demazure Hecke Category}

\begin{defn}
Let $\DD=(\cQ^!(\BGB),\ast)$ denote the Demazure Hecke category. The Demazure monad $\fd\in Alg(\DD)$ is the dualizing complex $\omega_{\BGB}$ with its natural algebra structure. 
\end{defn}

\begin{example}~\cite{AK2}
For any $G$-space $Z$, we have the $\DD$-module category $\cQ^!(Z/B)$.
The equivariants are given by $$\cQ^!(Z/B)^\DD\simeq \cQ^!(Z/B)_\fd\simeq \cQ^!(Z/G),$$
where the last equivalence follows from proper descent: since $Z/B\to Z/G$ is a base change of $pt/B\to pt/G$, the descent monad in $End(\cQ^!(Z/B))$ is given by the image of $\fd\in Alg(\DD)\to Alg(End(\cQ^!(Z/B))).$
\end{example}

\begin{remark} A corresponding result on the level of $K$-groups was proved in~\cite{HLS}, extending Demazure's description of the Weyl character formula (the case $Z=pt$).
\end{remark}

\begin{example}\label{intermediate} Specializing to $Z=\bG$, we find an action of $\DD$ by $\bG$-endomorphisms of $\cQ^!(\bG/B)$ -- i.e. (in the de Rham setting) a monoidal functor $$\DD=\cQ^!(B\backslash G/B)\to End_{\cD(G)}\cD(G)_B\simeq \cD_B(G)_B$$ to bi-weakly $B$-equivariant $\cD$-modules on $G$. This in turn comes (by passing to [weak] $B$-equivariants) from the induction functor $\cQ^!(G)\to \cQ^!(\bG)=\cD(G)$ associated to the group homomorphism $G\to \bG$, which is monoidal for convolution. The Demazure monad $\fd$ is taken to [the natural $B$ bi-equivariant structure on] the induced $\cD$-module $\cD^r_{G}$ on $G$.
\end{example}

We will be interested in transporting (lax) $\DD$-actions along adjunctions.
Given an adjunction
$$\xymatrix{
\phi:\cN\ar@{<->}[r]  & \cM: \psi}
$$ 
there is a natural lax monoidal structure on the functor $$\Psi:End(\cM)\to End(\cN),
\hskip.3in \Psi(T)=\psi\circ T \circ \phi$$
-- e.g., on the level of individual endomorphisms 
we have a map $$\Psi(T)\circ \Psi(U)= \psi\circ T \circ \phi \circ\psi \circ U \circ \phi \longrightarrow 
\psi\circ T \circ  U \circ \phi.=\Psi(T\circ U).$$
Thus if $\cM$ is a module category for $\DD$, i.e., we are given a monoidal functor $\DD\to End(\cM)$, then $\cN$ becomes a lax module category via the lax monoidal functor $\DD\to End(\cM)\to End(\cN)$.
In particular, we obtain a functor \begin{equation}\label{getting monads}
Alg(\DD)\longrightarrow Alg(End(\cN))\end{equation} from algebras in $\DD$ to monads on $\cN$.

We apply this construction to the adjunction

$$ \xymatrix{
p^!:\cN=\cQ^!(\bG/\bN H) \ar@{<->}[rr]  && \cM=\cQ^!(\bG/B):p^{!,R}
}$$
where $p^!$ forgets strong to weak $N$-equivariance (which preserves compact objects) 
and $p^{!,R}$ is its continuous right adjoint (informally, taking invariants for $U\fn$, rather than coinvariants as in the left adjoint).  

\begin{remark}\label{groupoid in correspondences}
Geometrically, such lax actions come from considering groupoid objects in the category of correspondences. For example, if $\cG\actson X$ is a groupoid and $X\to Y$ a morphism, then $\cG$ forms a groupoid object in the category of correspondences over $Y$ -- in particular the composition is given by a correspondence
$$\xymatrix{ \cG\times_Y \cG&\ar[r]\ar[l]\cG\times_X \cG& \cG}.$$
In the setting of Diagram~(\ref{BGB correspondence}) we see that $\cG=\bG/B\times_{pt/B}\BGB\actson X= \bG/B$ forms a groupoid in correspondences over $Y=\bG/\bN H$.
\end{remark}

\begin{prop} 
\begin{enumerate}
\item The categories $\cN=\cQ^!(\bG/\bN H)$, i.e., $\cQ^!(\gtil)$ (Dolbeault) and $\cD(G/N)_H$ (de Rham),
carry natural lax $\DD$-module structures. These actions come from the natural lax monoidal functor (induction)
$$\DD=\cQ^!(\BGB)\to \cQ^!(H\bN \backslash \bG/\bN H)=\cD_H(N\backslash G/N)_H$$
(where the last identification with the Hecke category is in the de Rham version).
\item Under the resulting functor~(\ref{getting monads}) the Demazure monad $\fd$ is taken to the Grothendieck-Springer and Weyl monads, respectively. Thus we find an identification of the Weyl monad with 
the universal Weyl sheaf $\cW$ of Lemma~\ref{universal Weyl sheaf}.
\item The resulting categories of modules for the Demazure monad are identified as $\cQ^!(\gtil)^\fd\simeq \cQ(\fg^\ast)$ and $\cD(G/N)_H^\fd\simeq U\fg\module$. 
\end{enumerate}
\end{prop}       

\begin{proof} The lax action is the result of the $(p^!,p^{!,R})$ adjunction above, or geometrically of the correspondence action from Remark~\ref{groupoid in correspondences}. By construction it commutes with the strong $G$-action, hence factors through a lax monoidal functor to the Hecke category, which is given by pushforward along the composition $$\BGB \to B\backslash \bG /B \to H\bN\backslash \bG/\bN H$$ -- the first pushforward (induction) encodes the action of $\DD$ on $\cQ^!(\bG/B)$ as in Example~\ref{intermediate}, while the second (imposing strong $N$ bi-equivariance, i.e., taking invariants for the $\fn$-actions) is only lax monoidal. 

The second assertion follows from Lemma~\ref{identify functors}, so ultimately from basechange for the diagram~(\ref{BGB correspondence}). In particular, the description of the action via induction provides an identification of the Weyl monad with the induction of the Demazure monad, so that the underlying sheaf is the sheaf of differential operators on $N\backslash G/N$: imposing strong $N$-biequivariance on the (weakly biequivariant) induced module $\cD^r$ on $G$ produces the induced module $\cD^r$ on $N\backslash G/N$, which is identified with $\cD$ by the Calabi-Yau structure on $G/N$.
Since we've already established monadicity in both Grothendieck-Springer and Beilinson-Bernstein situations, the third assertion follows.
\end{proof}

\begin{remark} There is a parallel comonadic story where we use (as in~\cite{AK2}) the quasicoherent version $\DD^*=(\cQ(\BGB),\ast)$ and the natural coalgebra object $\fd^*=\cO_{\BGB}$. The category $\cQ(\bG/\bN H)$ is endowed with an op-lax $\DD^*$-module structure and $\fd^*$ is taken to the Grothendieck-Springer and Beilinson-Bernstein comonads.
\end{remark}

\begin{remark}[Borel-Weil-Bott]
Passing to $G$-equivariant sheaves, we find the Borel-Weil-Bott adjunction (parabolic induction/restriction)
$$ \xymatrix{
\cQ^!(\bG\bs \bG/\bN H)=Rep(H) \ar@{<->}[r]  & \cQ^!(\bG\bs \bG/G) = Rep(G)}
$$
in both Dolbeault and de Rham settings, recovering the original appearance of Demazure operators in the Weyl character formula.
\end{remark}

\begin{remark}[Whittaker objects]  
Let us briefly indicate the case of Whittaker modules. Fix a non-degenerate character $\psi$ of $\fn$.
We may then pass to Whittaker objects in any strong $G$-category, i.e., objects strongly 
equivariant for $N$ against the character $\psi$ (see e.g.~\cite{dario}). When applied to $U\fg\module$, we obtain the category of Whittaker $U\fg$-modules, which by Kostant and Skryabin is identified with modules for $Z(U\fg)$. The same category arises in the classical limit, as sheaves on the Kostant slice $\fg^{\ast}//_{\psi} N\simeq \fc$. On the other hand, Whittaker $\cD$-modules on $G/N$ are identified with $\cD$-modules on the big cell $\cD(N\bs Nw_0HN/N\simeq \cD(H)$. Thus we find the following adjunction:$$ \xymatrix{
\cD(G/N)_H^{(N,\psi)}\simeq \cD(H)_H\simeq \cQ^!(\fh^\ast) \ar@{<->}[rr]  && U\fg\module^{(N,\psi)} \simeq \fZ\fg\module\simeq \cQ^!(\fh^\ast//W)}.
$$
In the classical limit we find the same adjunction, realized as the restriction of the Grothendieck-Springer resolution to the Kostant slice. 
We thus find that in the Whittaker setting the Beilinson-Bernstein or Demazure monads are taken to the nil-Hecke algebra $\BH$, which controls descent from $\fh^\ast$ to the coarse quotient $\fh^\ast//W$. 
\end{remark}

\end{document}